\RequirePackage[hyphens]{url}
\documentclass[11pt]{article}

\usepackage{endnotes}
\let\footnote=\endnote

\usepackage{tikz}
\usepackage{pgfplots}
\pgfplotsset{compat=1.18}
\usepgfplotslibrary{groupplots}
\usepgfplotslibrary{patchplots}
\usepgfplotslibrary{colorbrewer}
\usetikzlibrary{patterns}
\usepackage{ninecolors}
\usepackage{booktabs}
\usepackage{multirow}
\usepackage[export]{adjustbox}
\usepackage[symbol]{footmisc}

\usepgfplotslibrary{external}
\tikzset{external/system call={lualatex \tikzexternalcheckshellescape -halt-on-error -interaction=batchmode -jobname "\image" "\texsource"}}

\PassOptionsToPackage{hyphens}{url}
\usepackage{url,hyperref,subcaption}
\usepackage{multirow}
\usepackage{mathtools}
\usepackage{enumitem}

\newcommand\norm[1]{\left\lVert#1\right\rVert}
\newcommand\tp{^\top}
\newcommand{\lBrack}{[\![}
\newcommand{\rBrack}{]\!]}
\newcommand \Brack[1]{\lBrack #1 \rBrack}
\newcommand \ip[2]{\langle #1, #2 \rangle}

\newcommand{\eg}{e.g.}
\newcommand{\ie}{i.e.}

\DeclareMathOperator{\conv}{conv}
\DeclareMathOperator{\tr}{tr}

\usepackage{natbib}
 \bibpunct[, ]{(}{)}{,}{a}{}{,}

\usepackage{multibib}
\newcites{appendix}{References}

\usepackage{amsthm,amssymb}
\usepackage{setspace}
\usepackage[margin=0.8in]{geometry}
\usepackage{apptools}

\newtheorem{theorem}{Theorem}
\newtheorem{lemma}{Lemma}
\newtheorem{proposition}{Proposition}
\newtheorem{corollary}{Corollary}

\theoremstyle{definition}
\newtheorem{assumption}{Assumption}

\AtAppendix{\counterwithin{lemma}{section}}
\AtAppendix{\counterwithin{proposition}{section}}
\AtAppendix{\counterwithin{corollary}{section}}
\AtAppendix{\counterwithin{figure}{section}}
\AtAppendix{\counterwithin{table}{section}}
\AtAppendix{\counterwithin{equation}{section}}
\hypersetup{
           breaklinks=true,
           colorlinks=false, 
           pdfusetitle=true,
           hypertexnames=false,
        }
        
\title{Modeling Adversarial Wildfires for Power Grid Disruption}
\author{Matthew Brun\footnotemark[1] \and Xu Andy Sun\footnotemark[2] \footnotemark[4] \and Jean-Paul Watson\footnotemark[3]}
\date{\vspace{-3em}}

\begin{document}

\maketitle
\footnotetext[1]{Operations Research Center, Massachusetts Institute of Technology, Cambridge, MA (\href{mailto:brunm@mit.edu}{brunm@mit.edu}).}
\footnotetext[2]{Sloan School of Management, Massachusetts Institute of Technology, Cambridge, MA (\href{mailto:sunx@mit.edu}{sunx@mit.edu}).}
\footnotetext[3]{Lawrence Livermore National Laboratory, Livermore, CA (\href{mailto:watson61@llnl.gov}{watson61@llnl.gov}).}
\footnotetext[4]{Corresponding author.}

%%Abstract
\begin{abstract}
Electric power infrastructure faces increasing risk of damage and disruption due to wildfire.  Operators of power grids in wildfire-prone regions must consider the potential impacts of unpredictable fires.  However, traditional wildfire models do not effectively describe worst-case, or even high-impact, fire behavior.  To address this issue, we propose a mixed-integer conic program to characterize an adversarial wildfire that targets infrastructure while respecting realistic fire spread dynamics.  We design a wind-assisted fire spread set based on the Rothermel fire spread model and propose principled convex relaxations of this set, including a new relaxation of the inner product over Euclidean balls.  We present test cases derived from the recent Park, Eaton, and Palisades fires in California and solve models to identify the minimum time-to-outage of multiple-element contingencies and the maximum load shed associated with a sequence of element outages caused by a realistic wildfire.  We use the minimum time-to-outage values to screen contingencies and construct security-constrained optimal power flow models that promote operational resilience against wildfire.
\end{abstract}


%%Introduction
\section{Introduction}

Wildfires pose a growing threat to human life, property, and infrastructure, due to increases in severity and annual burned area \citep{wasserman2023climate, nifc2025statistics}.  Fire jeopardizes the resilience of power infrastructure, representing a leading cause of large-scale blackouts which affect hundreds of thousands of customers \citep{hines2008trends}.  Due to wildfire, utilities operating the California power grid incurred estimated costs exceeding \$700 million between 2001 and 2016 \citep{dale2018assessing}.  In one example, the 2017 Thomas Fire in southern California caused power outages for over 260,000 people; repairs to power grid equipment after this fire cost \$49 million \citep{cpuc2017investigation}.  In 2019, a wildfire in Maui, Hawaii damaged transmission lines and threatened a 230 MW generation plant, necessitating power conservation efforts \citep{maui2019fire}.

Several strategies have been proposed to improve the resilience of the power grid to wildfires \citep{nazaripouya2020power}.  These strategies include proactive investments and operational changes.  Proactive investment involves upgrading equipment, undergrounding power lines, managing vegetation, and inspecting infrastructure, to reduce the chance of equipment damage during a wildfire.  Alternatively, changes to operational strategies, including improved dispatch, optimal load shedding, and network reconfiguration, can help alleviate the effects of wildfire on power delivery.  However, effective implementation of these strategies requires an underlying model to evaluate the risk and impact of wildfires.

Wildfire spread is highly unpredictable due to uncertainty in weather patterns, locations of ignition, and responses to intervention \citep{thompson2016uncertainty}.  Many models of wildfire behavior have been proposed; see \cite{sullivan2009physical,sullivan2009empirical,sullivan2009simulation} for an overview.  Traditionally, these models span three classes.  First, \emph{physical} models are based in detailed descriptions of the underlying chemical and physical dynamics involved in combustion (\eg, \citealt{linn1997firetec}).  Second, \emph{empirical} models estimate the rate of fire spread from experimental or historical data (\eg, \citealt{rothermel1972mathematical}).  Third, \emph{simulated analogues} propagate pointwise or cell-based fire representations, based on some underlying physical or empirical model (\eg, \citealt{finney1998farsite}).  Alternatively, machine learning has been used to predict wildfire spread from rich spatial datasets \citep{huot2022next}.  These approaches vary widely in accuracy, detail, spatial and temporal resolution, and requisite computational time and memory.

The increase in wildfire prevalence and severity mandates consideration of the most extreme potential impacts of fires to the power grid.  Of particular concern are the concurrent outages of multiple components (\ie, $N - k$ contingencies) and the timing and sequencing of these outages.  In this work, we take inspiration from the field of robust optimization and treat wildfire as an adversary.  Under the empirical model of \cite{rothermel1972mathematical}, we design an optimization model that characterizes the worst-case wildfire ignition and spread behavior.  This perspective permits identification of realistic fire behavior that imparts the most harm to power infrastructure.  Our adversarial model quantifies how quickly a multi-component outage can occur and identifies outage sequences that cause the largest disruption to power delivery.

\subsection{Literature Review}

Much attention has been devoted to the design and operation of power systems under the threat of wildfire.  Here, we focus specifically on works that address the uncertain impact of natural events on the power grid.

\cite{choobineh2015vulnerability} and \cite{ansari2015optimal} propose a framework for adjusting power line thermal ratings due to the influence of a nearby wildfire.  \cite{mohagheghi2015optimal} solve a stochastic optimal power flow problem over these dynamic line ratings, treating wind and rate of fire spread as random parameters.  \cite{trakas2017optimal} extend this approach to the resilient configuration and operation of distribution grids.  Other approaches model wildfire disruption of power equipment as random events and solve stochastic optimization models to mitigate these disruptions.  \cite{yang2024multi, yang2024multistage} generate outage scenarios using a rule-based cellular automaton to simulate wildfire spread, and solve multi-stage stochastic dispatch models to minimize the expected cost of wildfire interference.  \cite{estrada2025dynamic} consider decision-dependent uncertainty in a multi-stage stochastic transmission switching problem, where line faults occur due to interactions with nearby fires and increase in probability with the amount of power flow on the line.  This work constructs a wildfire scenario tree by simulations of a cellular automaton, initialized with historical wildfire data.  \cite{abdelmalak2022enhancing} describe a Markov decision process for grid operations, where the states represent a sequence of element outages due to a wildfire, generated by element-wise disruption probabilities.

An alternative perspective is to solve a robust optimization problem over an uncertainty set that summarizes reasonable disruptions under a natural disaster.  \cite{tapia2021robust} construct a budget uncertainty set for element outages based on the outcome of a wildfire simulation, and solve a robust dispatch problem.  In the context of hurricanes, \cite{yuan2016robust} use expert knowledge to design a multi-stage and multi-zone uncertainty set for line outages, taking into account spatial and temporal behavior; this uncertainty set is used in a robust network planning problem for the distribution system. To address variability in renewable power generation, \cite{lorca2014adaptive} solve a robust multi-stage dispatch problem that considers the spatial and temporal correlation of wind, using an uncertainty set derived from historical data.

Another common optimization problem at the intersection of wildfire and power systems optimization is the optimal power shutoff problem \citep{rhodes2020balancing}.  A \emph{public safety power shutoff} (PSPS) is an intentional de-energization of power lines, intended to reduce the risk of wildfire ignition from faults; these shutoff decisions can be optimized to balance the risk of ignition due to active line failures against the cost of power delivery and load shedding after de-energization.  \cite{zhou2024mitigating} treat wildfire risk and the resulting PSPS as a stochastic event and optimize the network topology and dispatch as a two-stage stochastic program.  \cite{bayani2023resilient} propose a budget uncertainty set for wildfire risk, electricity demand, and renewable generation, and solve a robust capacity expansion model. \cite{taylor2025optimizing} optimize power line undergrounding decisions with stochastic and robust perspectives on wildfire risk and PSPS events.

Most prior work utilizes simulation and historical data to capture uncertainty in wildfire evolution and interaction with the power grid.  However, these approaches may not accurately capture rare, high-impact, or worst-case outcomes, which are of particular importance when designing and operating a resilient power grid.  In this work, we address this issue by characterizing wildfire uncertainty through the lens of an empirical fire spread model.  To our knowledge, such an approach has not been proposed in the literature.

\subsection{Contributions}

In this work, we propose a novel adversarial model for wildfire spread.  This model maximizes the impact of a newly-ignited wildfire on power grid elements, subject to realistic wind-assisted fire spread dynamics.  We use this model to characterize how soon after ignition multiple-element outages can occur and to determine feasible sequences of element outages that incur the most damage.  The outputs of this model are used to build a security-constrained optimal power flow model that encourages resilient power grid operation against realistic contingencies.

Specifically, our contributions are as follows:
\begin{enumerate}
    \item We formulate a novel mixed-integer conic program (MICP) that maximizes the impact of a realistic, adversarial wildfire on power grid infrastructure;
    \item We propose a two-dimensional wind-assisted fire spread set based on the empirical Rothermel fire spread model \citep{rothermel1972mathematical}.  For computational tractability, we propose conic relaxations of this set, including a novel second-order conic (SOC) inequality that gives a convex relaxation of the inner product over Euclidean balls;
    \item We design experimental settings aligned with three California wildfires: the 2024 Park fire, the 2025 Eaton fire, and the 2025 Palisades fire.  We calibrate our wildfire model and a 3,928-bus optimal power flow model with data from California at the times of the fires.  We propose an approach to aggregate spatial wildfire risk data into polyhedral regions, which are given as inputs to the adversarial wildfire model;
    \item We apply our adversarial model to characterize the minimum time-to-outage for a set of high-impact $N-k$ contingencies and to determine the maximum load shed from a sequence of component disruptions due to the realistic spread of a newly-ignited wildfire; and
    \item We use the minimum time-to-outage values as a contingency screening parameter and solve a detailed security-constrained optimal power flow model, encouraging operational resilience against reasonable contingencies while reducing the total number of contingencies considered.
\end{enumerate}

Although this work presents an adversarial wildfire model in the context of power grid operations, our model formulation is highly generalizable and compatible with other operational problems under the threat of wildfire.  These settings include resource deployment for fighting fires \citep{ntaimo2013simulation} and fuel management for fire mitigation \citep{minas2014spatial}.  Our model can help characterize worst-case wildfire behavior to improve strategies for wildfire prevention, suppression, mitigation, and threat evaluation.

The remainder of this paper is organized as follows.  Section~\ref{sec:fireModel} formulates an optimization model to characterize adversarial wildfires.  Section~\ref{sec:spreadModel} proposes a wildfire spread set, based on the Rothermel spread model, and introduces high-quality convex relaxations of this set.  Section~\ref{sec:micpFormulation} gives an explicit mixed-integer conic formulation of the adversarial model.  Section~\ref{sec:opf} presents our security-constrained optimal power flow model.  Section~\ref{sec:data} describes the experimental settings, power grid and wildfire data, and an approach to summarize wildfire risk values into polyhedral regions.  Section~\ref{sec:experiments} reports computational results on the adversarial wildfire and optimal power flow models.  Section~\ref{sec:conclusion} concludes the paper.

\subsection{Notation}

We denote $\Brack{\cdot} := \{1,\dots,\cdot\}$.  Let $0_n$ be the zero vector in $\mathbb{R}^n$ and $\mathbf{e}_n$ be the vector of ones in $\mathbb{R}^n$.  Let $I_n$ be the identity matrix in $\mathbb{R}^{n \times n}$.  The set of symmetric $n \times n$ matrices is written as $\mathbb{S}^{n}$.  The indicator function for set $A$ is given by $\mathbb{I}_A(\cdot)$.  The dot product of vectors $x$ and $y$ is given by $\ip{x}{y}$ and the Euclidean norm by $\norm{x}_2 = \sqrt{\ip{x}{x}}$.  Similarly, $\ip{A}{B}$ gives the Frobenius inner product of matrices $A$ and $B$, and $A \otimes B$ gives the Kronecker product.  We write the Euclidean ball with center $c \in \mathbb{R}^n$ and radius $r \in \mathbb{R}$ as $\mathcal{B}_r(c) := \{x \in \mathbb{R}^n \,:\, \norm{x - c}_2 \leq r\}$.  We assume that upper bounding power constraints of the form $x \leq y^r$, with variables $(x,y)$ and parameter $r$, implicitly enforce the nonnegativity of the argument $y$.

%%An Optimization Model for Adversarial Wildfires
\section{An Optimization Model for Adversarial Wildfires}
\label{sec:fireModel}

We model the disruption of power infrastructure due to the spread of a newly ignited wildfire.  Let $\mathcal{E}$ give the set of power grid elements threatened by the fire, where $\mathcal{L}_e \subset \mathbb{R}^2$ defines the geometry of grid element $e \in \mathcal{E}$.  For generators or buses, $\mathcal{L}_e$ is a point, and for power lines, $\mathcal{L}_e$ is a line segment between the start and end buses of the line.  In both cases, $\mathcal{L}_e$ is a polyhedral set.  We select $\mathcal{E}$ to be a subset of grid elements that generate a high-impact contingency when outaged (see Section~\ref{subsec:gridData}).

We consider a temporal discretization, where $T$ gives the number of periods (\eg, hours) and $\mathcal{T} := \{0\} \cup \Brack{T}$ gives the set of periods.  The set $\mathcal{S}(w,x) \subset \mathbb{R}^2$ gives the area that is reachable in one period by a wildfire currently burning at point $x \in \mathbb{R}^2$, which depends on the wind velocity $w \in \mathbb{R}^2$.  This ``spread set'' characterizes our model of wildfire spread dynamics; Section~\ref{sec:spreadModel} provides principled constructions for several variants of this set.

We use a binary outage model, where $o_{et} \in \{0,1\}$ is the indicator of whether element $e \in \mathcal{E}$ is outaged by wildfire in period $t \in \mathcal{T}$.  We require
\begin{equation}
    \label{eq:outageCondition}
    o_{et} = 1 \implies \{x_t \in \mathcal{L}_e : x_{t'+1} \in \mathcal{S}(x_{t'},w_{t'})\ \forall t' \in \mathcal{T}\} \neq \emptyset,
\end{equation}
where $w_t$ gives the wind velocity in period $t$ and $x_0$ gives the ignition point of the wildfire.  That is, an element is reachable by wildfire in period $t$ if there is a sequence of points contained in the spread sets where the $t$-th element of the sequence intersects the element geometry.  An element can only be outaged at period $t$ if it is reachable by wildfire in that period; however, an element is not required to be outaged if it is reachable.

The adversarial setting optimizes over the wind velocity, the wildfire ignition point, and the sequence of wildfire points to cause a worst-case impact on the power grid.  Let $x_0$ denote the ignition point and $w_t$ denote the wind velocity at time $t \in \mathcal{T}$.  The adversarial wildfire spread model is as follows:
\begin{subequations}
    \label{abstractOpt}
    \begin{align}
        \max_{o,w,x} \quad & \mathrlap{\sum_{\mathcal{E}' \subseteq \mathcal{E}}\ \sum_{t \in \mathcal{T}} c_t(\mathcal{E}') \ \prod_{e \in \mathcal{E}'} o_{et} \prod_{e \in \mathcal{E} \setminus \mathcal{E}'} (1 - o_{et})} \label{eq:obj}\\
        \mathrm{s.t.} \quad & x_{e,t+1} \in \mathcal{S}(w_t,x_{et}) \quad && \forall e \in \mathcal{E},\ t \in \mathcal{T}, \label{eq:constrSpread}\\
        & o_{et} \leq \mathbb{I}_{\mathcal{L}_e}(x_t) \quad && \forall e \in \mathcal{E},\ t \in \mathcal{T}, \label{eq:constrOutage}\\
        & o_{et} \leq o_{e,t+1} \quad && \forall e \in \mathcal{E},\ t \in \mathcal{T}, \label{eq:constrMonotonic}\\
        & x_0 = x_{e0} \quad && \forall e \in \mathcal{E}, \label{eq:constrIgnition}\\
        & w_t \in \mathcal{B}_\varepsilon(\overline{w}_t) \quad && \forall t \in \mathcal{T}, \label{eq:constrWind}\\
        & o_{et} \in \{0,1\} \quad && \forall e \in \mathcal{E},\ t \in \mathcal{T} \label{eq:constrBin}.
    \end{align}
\end{subequations}
To model the outage condition \eqref{eq:outageCondition}, we introduce variables $x_{et} \in \mathbb{R}^2$ for each element $e \in \mathcal{E}$, which represent points that have been reached by wildfire at time $t$.  The locations of these points interact through the wind velocity variables $w_t$ and the ignition point variable $x_0$.  Constraint~\eqref{eq:constrSpread} enforces pointwise wildfire spread dynamics, and constraint~\eqref{eq:constrOutage} defines the outage indicator logic; namely, an element can only be outaged once it is reached by its corresponding sequence of wildfire points.  Together, these constraints imply the element outage logic \eqref{eq:outageCondition}.  Constraint~\eqref{eq:constrMonotonic} requires that, once an element is outaged, it remains outaged.  Constraint~\eqref{eq:constrIgnition} models ignition, ensuring that the wildfire point sequences all start at the ignition location.  The wind velocity is treated as uncertain and can vary from a nominal prediction $\overline{w}_t$ by magnitude $\varepsilon \geq 0$ under constraint~\eqref{eq:constrWind}.

For every subset $\mathcal{E}' \subseteq \mathcal{E}$ of the power grid elements, the function $c_t(\mathcal{E}')$ assigns a weight to the outage of exactly the elements $\mathcal{E}'$ in period $t$.  Then, the objective~\eqref{eq:obj} maximizes the weighted sum of the outage sequence.  This objective is flexible; for instance, the functions $c_t$ can measure the impact of each outage by some metric (\eg, incurred load shed).  Alternatively, with 
\begin{equation}
    \label{eq:minTimeObj}
    c_t(\mathcal{E}') = \begin{cases}
        1 & \mathcal{E'} = \mathcal{E}\\
        0 & \mathrm{otherwise}
    \end{cases} \quad \forall \mathcal{E}' \subseteq \mathcal{E},\ t \in \mathcal{T},
\end{equation}
the objective measures the number of periods in which all elements are outaged.  Then, the optimization identifies the \emph{minimum time-to-outage}, the soonest period in which all elements $\mathcal{E}$ can be reached by wildfire.

As the variables $o_{et}$ are binary and the geometries $\mathcal{L}_e$ are polyhedral, \eqref{abstractOpt} can be written as an MICP if $\mathcal{S}(w,x)$ are mixed-integer conic representable sets.  We provide explicit reformulations of this problem as an MICP in Section~\ref{sec:micpFormulation}. Before this, we introduce model-driven characterizations of the wildfire spread sets $\mathcal{S}(w,x)$ in Section \ref{sec:spreadModel}.

%%Pointwise Wildfire Spread Models
\section{Pointwise Wildfire Spread Models}
\label{sec:spreadModel}

The wildfire spread dynamics which appear in model~\eqref{abstractOpt} depend on the definition of the spread sets $\mathcal{S}(w,x)$.  This section introduces several choices for this set; these constructions are some of the main technical contributions of this work.  Section~\ref{subsec:rothermel} introduces a convex conic spread set inspired by the Rothermel fire spread model.  However, this spread set is defined by an inner product with the wind vector $w$, which introduces nonconvexity to the adversarial model \eqref{abstractOpt}.  To handle this nonconvexity, Section~\ref{subsec:iprelax} proposes convex relaxations of the inner product, including a novel second-order cone relaxation of the inner product over Euclidean balls.  Section~\ref{subsec:ballrelax} describes an alternative resolution of this nonconvexity by relaxing the base set in a principled way, so that the resulting set is a Euclidean ball.  Section~\ref{subsec:chooserelax} discusses the relative quality of the relaxations.

\subsection{The Rothermel Fire Spread Model}
\label{subsec:rothermel}

The Rothermel fire spread model \citep{rothermel1972mathematical, andrews2018rothermel} is a one-dimensional rate-of-spread fire model.  Given inputs describing the fuel bed (\ie, the mix of organic material which fuels the fire), wind speed, and slope, the model predicts the \emph{rate of spread} $R$, that is, the speed at which the fire expands.  The Rothermel model computes the rate of spread as
\begin{equation}
    \label{eq:rothermel}
    R = \frac{F^{\mathrm{src}}}{E^{\mathrm{ig}}} (1 + \phi_w + \phi_s),
\end{equation}
where $F^{\mathrm{src}}$ is the \emph{thermal flux} of the fire, the amount of heat produced per area per minute, and $E^{\mathrm{ig}}$ is the \emph{ignition energy density} of the fuel bed, the amount of energy required to ignite a unit volume of fuel.  These parameters depend on the composition and structure of the fuel bed.  We define the \emph{base rate of spread} $V : = F^{\mathrm{src}} / E^{\mathrm{ig}}$, which gives the rate of spread without any influence from wind or slope.  The elements $\phi_w$ and $\phi_s$ are flux multipliers due to the impact of wind and slope, respectively.  Under our model calibration, the impact of slope on the rate of spread is much smaller than the impact of wind: the flux multiplier due to a 20 mi/hr wind is over 6x as large as the flux multiplier due to a $30^\circ$ slope.  For this reason and for modeling tractability, we ignore the impact of slope and take $\phi_s = 0$, focusing primarily on the wind effect.  The Rothermel model defines 
\begin{equation}
    \phi_w = C\cdot s^{B},
\end{equation}
where $s$ gives the wind speed.  The coefficient $C$ and exponent $B$ are empirically-derived parameters which vary with two characteristics of the fuel bed: the \emph{packing ratio} (\ie, the proportion of the fuel bed volume composed of fuel particles) and the \emph{surface-area-to-volume ratio} of the fuel particles.

In a two-dimensional model, the wind velocity may not align with the direction of fire spread.  To handle this issue, we project the wind velocity $w$ onto the direction of spread $d$, which gives a spread-aligned wind speed.  Additionally, we model the geographic dependence of the base rate of spread $V$ by a set of regions.  Let $\mathcal{R}$ be the region index set, $\mu_r \in (0,1]$ be the rate of spread multiplier associated with region $r \in \mathcal{R}$, and $\mathcal{P}_r \subset \mathbb{R}^2$ be the geometry of the region.  We use polyhedral regions $\mathcal{P}_r$ which form a partition of the area of study.  A data-driven method for constructing these regions and multipliers is given in Section~\ref{sec:data}.

Let the function $V(x)$ compute the base rate of spread associated with the region containing $x$:
\begin{equation}
    V(x) := \sum_{r \in \mathcal{R}} \mu_r \cdot V\cdot\mathbb{I}_{\mathcal{P}_r}(x).
\end{equation}
We now introduce the \emph{Rothermel-inspired spread set}
\begin{equation}
    \label{eq:rothermelSpread}
    \mathcal{S}^{\mathrm{r}}(w,x) := \left \{d + x : \norm{d}_2 \leq V(x) \left (1 + C \left (\frac{\langle d,w \rangle }{\norm{d}_2} \right ) ^B \right ) \right \}.
\end{equation}
In this definition and elsewhere, we take  $0 / 0 = 0$, so that $d = 0_2$ is feasible.  This set bounds the magnitude of fire spread $\norm{d}_2$ by the rate of spread $R$ under the Rothermel model \eqref{eq:rothermel}, where the wind speed is given by the projection of the wind vector $w$ onto the direction of spread $d$ and the base rate of spread is adjusted by the function $V(x)$.

\begin{proposition}
    \label{prop:conicRothermel}
    Fix some $\{x,w\} \subset \mathbb{R}^2$.  Let $B \in  [0,1]$ and $(C,V(x)) \geq 0$.  Then, the set $\mathcal{S}^{\mathrm{r}}(w,x)$ can be written as the projection of a power conic set; specifically, 
    \begin{equation}
        \label{eq:conicRothermel}
        \mathcal{S}^{\mathrm{r}}(w,x) = \left \{ d + x \,:\, \exists \gamma \ \mathrm{s.t.} \  \begin{array}{rl}
            \gamma_1 & \geq \norm{d}_2\\
            \gamma_1 & \leq (\gamma_2 + \gamma_3)^\frac{1}{1+B}\\
            \gamma_2 & \leq V(x) (\gamma_1)^B\\
            \gamma_3 & \leq C V(x) (\ip{d}{w})^B
        \end{array} \right \}.
    \end{equation}
\end{proposition}

\begin{proof}
    Observe that, by multiplying by $(\norm{d}_2)^B$ and raising to the power of $1/(1+B)$, the constraint which defines $\mathcal{S}^{\mathrm{r}}(w,x)$ is equivalent to 
    \begin{equation}
        \label{eq:modifiedRothermel}
        \norm{d}_2 \leq \left ( V(x) (\norm{d}_2)^B + C V(x)  (\ip{d}{w})^B\right )^{\frac{1}{1+B}}.
    \end{equation}
    For $d$ satisfying \eqref{eq:modifiedRothermel}, taking 
    $\gamma_1 = \norm{d}_2$ and $(\gamma_2,\gamma_3)$ equal to their upper bounds gives a feasible solution to the conic description \eqref{eq:conicRothermel}.  For $(d,\gamma)$ that satisfy the conic description \eqref{eq:conicRothermel}, by application of the inequalities in $\gamma$ and the monotonicity of the power functions, we have
    \begin{equation}
        \label{eq:rothermelConicProof}
        \begin{aligned}
            \norm{d}_2 \leq \gamma_1 = \frac{(\gamma_1)^{1+B}}{(\gamma_1)^B} \leq V(x) + \frac{C V(x) (\ip{d}{w})^B}{(\gamma_1)^B}.
        \end{aligned}
    \end{equation}
    As $B\geq 0$, the function $(1/\gamma_1)^B$ is decreasing in $\gamma_1$. Then, $\norm{d}_2 \leq \gamma_1$ and \eqref{eq:rothermelConicProof} imply the definition of the set $\mathcal{S}^{\mathrm{r}}(w,x)$.
    Therefore, the two representations are equivalent.  As the exponents $B$ and $1/(1+B)$ fall on $[0,1]$ and the coefficients $C$ and $V(x)$ are nonnegative, the constraints in $\gamma$ are power cone constraints; that is, they are described by conic sets with the form $\{(x,y) \in \mathbb{R}^{n+1} \,:\, \norm{x}_2 \leq y^r\}$, where $r \in [0,1]$.
\end{proof}

Proposition~\ref{prop:conicRothermel} shows that the Rothermel spread set $\mathcal{S}^{\mathrm{r}}(w,x)$ can be reformulated with power cone constraints when $w$ and $x$ are fixed and $B \leq 1$.  However, choosing $\mathcal{S}(w,x) = \mathcal{S}^{\mathrm{r}}(w,x)$ in problem \eqref{abstractOpt} introduces nonconvexity, as the problem optimizes jointly over the spread set and the arguments $x$ and $w$. In particular, the nonconvexity due to the dependence of $V(x)$ on the location $x$ through the indicator function $\mathbb{I}_{\mathcal{P}_r}(x)$ can be handled by the introduction of binary variables, as described in Section~\ref{sec:micpFormulation}.  However, the nonconvexity of the inner product $\langle d, w\rangle$ between the direction of fire spread and the wind vector cannot be easily reformulated.  To handle this interaction, we propose two convex relaxations in the $(w,x)$ space.  In Section~\ref{subsec:iprelax}, we introduce a relaxation that addresses this nonconvexity under the assumption that $B \leq 1$; in Section~\ref{subsec:ballrelax}, we propose an alternative relaxation for the case where $B \geq 1$.  These relaxations maintain the conservatism of the adversarial model by overestimating the impact of the worst-case fire.

\subsection{The Inner Product Spread Set}
\label{subsec:iprelax}

The first approach to address the nonconvexity is to directly relax the inner product.  As in problem~\eqref{abstractOpt}, we assume that the wind vector $w \in \mathcal{B}_{\varepsilon}(\overline{w}_t)$ for some $t$.  Denote by
\begin{equation*}
    \overline{W} := \max_{t \in \mathcal{T}}\ \norm{\overline{w}_t}_2 + \varepsilon \quad \text{and} \quad \overline{R} := V \cdot \left (1 + C \overline{W}^B \right )
\end{equation*}
the bounds on the magnitude of the wind vector $w$ and the rate of spread, respectively.  Then, for any $t$ and $w \in \mathcal{B}_\varepsilon(\overline{w}_t)$, it holds that $\mathcal{S}(w,x) \subseteq \mathcal{B}_{\overline{R}}(x)$.  Over the Euclidean balls which contain $w$ and $d$, we define the hypograph of the inner-product $\langle d, w\rangle$ in dimension $n$ as
\begin{equation}
    \mathcal{Z}^{\mathrm{IP}}_n := \left \{(d,w,z) \in \mathbb{R}^{2n+1} \,:\, \begin{array}{rl}
        z & \leq \ip{d}{w}\\
        d & \in \mathcal{B}_{\overline{R}}(0_n)\\
        w & \in \mathcal{B}_\varepsilon(0_n)
    \end{array} \right \}.
\end{equation}
Here, we are concerned only with the hypograph of the inner product, as the Rothermel spread set~\eqref{eq:rothermelSpread} upper bounds $\norm{d}_2$ by an increasing function in $\ip{d}{w}$.
Additionally, for $x \in \mathcal{S}(w,x')$ and $w \in \mathcal{B}_\varepsilon(\overline{w}_t)$, we have that $\ip{x - x'}{w} = \ip{d}{\overline{w}_t + w'}$, where $w' := w - \overline{w}_t \in \mathcal{B}_\varepsilon(0_n)$ and $d \in \mathcal{B}_{\overline{R}}(0_n)$.  As $\ip{d}{\overline{w}_t}$ is affine in $(d,w)$, it does not hinder convexity, and we only need to consider the bilinear portion $\ip{d}{w'}$.  We proceed by constructing convex relaxations $\overline{\mathcal{Z}} \supseteq \mathcal{Z}^{\mathrm{IP}}_n$.  

\subsubsection{The Shor Relaxation}

A standard convex relaxation of a nonconvex quadratic set is the Shor semidefinite relaxation \citep{shor1987quadratic}.  \cite{wang2024semidefinite} provide sufficient conditions under which this relaxation gives the convex hull of the quadratic set; Theorem~\ref{thm:shorExact} presents one of these results.

\begin{theorem}[\citealt{wang2024semidefinite} Sec.~4.1]
    \label{thm:shorExact}
    For $\{A_i\}_{i=0}^m \subset \mathbb{S}^n$ and $c \in \mathbb{R}^m$, let 
    \begin{equation}
        \label{eq:generalQuadratic}
        \mathcal{Z} := \left \{(y,z) \,:\, \begin{array}{rlr}
            \ip{y}{A_0 y} & \leq 2z &\\
            \ip{y}{A_i y} & \leq c_i & \forall i \in \Brack{m}
        \end{array} \right \}
    \end{equation} and 
    \begin{equation}
        \label{eq:shorQuadratic}
        \overline{\mathcal{Z}} := \left \{(y,z) \,:\, \exists Y \succeq y y \tp \ \mathrm{s.t.} \begin{array}{rlr}
            \ip{A_0}{Y} & \leq 2z &\\
            \ip {A_i}{Y} & \leq c_i & \forall i \in \Brack{m}
        \end{array} \right \}.
    \end{equation}
    If there exists some $\pi \geq 0$, $\{\mathbb{A}_i\}_{i=0}^m \subset \mathbb{S}^r$, and $k \geq m$ such that $\sum_{i=0}^m \pi_i A_i \succ 0$ and $A_i = \mathbb{A}_i \otimes I_k$, then $\conv(\mathcal{Z}) = \overline{\mathcal{Z}}$.
\end{theorem}

Theorem~\ref{thm:shorExact} establishes that, if a dual strict feasibility condition is met and the quadratic set is sufficiently symmetric, the Shor relaxation gives the convex hull.  In Corollary~\ref{corr:ipCH}, we demonstrate that this result applies to the inner product hypograph over Euclidean balls.  Note that Corollary~\ref{corr:ipCH} does not apply when $n=1$; in this case, $\mathcal{Z}^{\mathrm{IP}}_1$ reduces to the product of one-dimensional variables over a box, where the convex hull is given by the polyhedral McCormick envelope \citep{mccormick1976computability}.  

\begin{corollary}
    \label{corr:ipCH}
    For $n \geq 2$, it holds that $\conv (\mathcal{Z}^{\mathrm{IP}}_n) = \mathcal{Z}^{\mathrm{CH}}_n$, where
    \begin{equation}
        \mathcal{Z}^{\mathrm{CH}}_n := \left \{ (d,w,z) \,:\, \begin{array}{rl}
            \exists \begin{bmatrix}
                D & Z \tp\\
                Z & W
            \end{bmatrix} & \succeq \begin{bmatrix}
                d\\w
            \end{bmatrix} \begin{bmatrix}
                d\\w
            \end{bmatrix} \tp\\
            \tr (Z) & \geq z\\
            \tr (D) & \leq \overline{R}^2\\
            \tr (W) & \leq \varepsilon^2
        \end{array} \right \}.
    \end{equation}
\end{corollary}

\begin{proof}
    Let $y = \begin{bmatrix}
        d & w
    \end{bmatrix} \tp \in \mathbb{R}^{2n}$ and 
    \begin{equation*}
        \mathbb{A}_0 := -\begin{bmatrix}
            0 & 1\\
            1 & 0
        \end{bmatrix},\ \mathbb{A}_1 := \begin{bmatrix}
            1 & 0\\
            0 & 0
        \end{bmatrix},\text{ and } \mathbb{A}_2 := \begin{bmatrix}
            0 & 0\\
            0 & 1
        \end{bmatrix}.
    \end{equation*}
    Construct the sets $\mathcal{Z}$ and $\overline{\mathcal{Z}}$ as in \eqref{eq:generalQuadratic}~and~\eqref{eq:shorQuadratic}, with $k = n$ and $A_i := \mathbb{A}_i \otimes I_n$. 
    Then, 
    \begin{equation*}
        \begin{aligned}
            \ip{y}{A_1 y} \leq \overline{R}^2 & \iff d \in \mathcal{B}_{\overline{R}}(0_n),\\
            \ip{y}{A_2 y} \leq \varepsilon^2 & \iff w \in \mathcal{B}_{\varepsilon}(0_n), \text{ and}\\
            \ip{y}{A_0 y} \leq 2z & \iff -z \leq \ip{d}{w}.
        \end{aligned}
    \end{equation*}
    We then have that $\mathcal{Z} = \{ (d,w,-z) \,:\, (d,w,z) \in \mathcal{Z}^{\mathrm{IP}}_n\}$ and that $\overline{\mathcal{Z}} = \{(d,w,-z) \,:\, (d,w,z) \in \mathcal{Z}^{\mathrm{CH}}_n\}$.  
    Take $\begin{bmatrix}
        \pi_0 & \pi_1 & \pi_2
    \end{bmatrix} = \begin{bmatrix}
        0 & 1 & 1
    \end{bmatrix}$.  Then, $\sum_{i=0}^2 \pi_i A_i = I_{2n} \succ 0$, and $k = n \geq 2 = m$, so the conditions of Theorem~\ref{thm:shorExact} hold and $\conv(\mathcal{Z}) = \overline{\mathcal{Z}}$. The result follows from the fact that the convex hull of an affine transformation of a set is exactly the affine transformation of the convex hull of the set.
\end{proof}

Although we directly characterize $\conv(\mathcal{Z}^{\mathrm{IP}}_n)$ via the Shor relaxation, the set $\mathcal{Z}^{\mathrm{CH}}_n$ is modeled with semidefinite constraints.  If this set is used to construct $\mathcal{S}(w,x)$, the problem~\eqref{abstractOpt} becomes a mixed-integer semidefinite program, which is beyond the current reach of commercial solvers.  For this reason, we continue to seek convex relaxations with structures that are compatible with mixed-integer optimization.

\subsubsection{Second-Order Minor Constraints}
\label{subsubsec:SOMIP}
Inspired by the Shor relaxation $\mathcal{Z}^{\mathrm{CH}}_n$, we introduce a SOC relaxation.  Let variables $\delta \in \mathbb{R}^n$, $\omega \in \mathbb{R}^n$, and $\zeta \in \mathbb{R}^n$ correspond to the diagonal elements of the matrices $D$, $W$, and $Z$, respectively.  Then, the three trace constraints of $\mathcal{Z}^{\mathrm{CH}}_n$ can be written as 
\begin{equation}
    \label{eq:2mSOCSum}
    \sum_{i=1}^n \delta_i \leq \overline{R}^2,\quad \sum_{i=1}^n \omega_i \leq \varepsilon^2, \quad \text{and} \quad \sum_{i=1}^n \zeta_i \geq z.
\end{equation}

By the positive semidefinite condition of the Schur complement \citep[\eg,][App.~A.5]{boyd2004convex}, we can equivalently write the semidefinite constraint of $\mathcal{Z}^{\mathrm{CH}}_{n}$ as 
\begin{equation*}
    \begin{bmatrix}
        1 & d \tp & w \tp\\
        d & D & Z \tp\\
        w & Z & W
    \end{bmatrix} \succeq 0.
\end{equation*}
The nonnegativity of the $2 \times 2$ minors of this matrix implies for all $i \in \Brack{n}$ that
\begin{equation}
    \label{eq:2mSOCMinor}
    \delta_i \geq (d_i)^2,\quad \omega_i \geq (w_i)^2, \quad \text{and} \quad \delta_i \omega_i \geq \zeta_i^2,
\end{equation}
where these constraints are taken from the $2 \times 2$ submatrices with index pairs $(1,i+1)$, $(1,2i+1)$, and $(i+1,2i+1)$, respectively.  All three constraints are SOC constraints and thus are compatible with optimization over mixed-integer variables.  We denote the relaxation generated by these constraints as 
\begin{equation}
    \mathcal{Z}^{\mathrm{2M}}_n := \{(d,w,z) \,:\, \exists (\delta,\omega,\zeta) \text{ s.t. } \eqref{eq:2mSOCSum},\ \eqref{eq:2mSOCMinor}\}.
\end{equation}
As this set is implied by the Shor relaxation, we have $\mathcal{Z}^{\mathrm{2M}}_n \supseteq \mathcal{Z}^{\mathrm{CH}}_n \supseteq \mathcal{Z}^{\mathrm{IP}}_n$.

For each element $e \in \mathcal{E}$ and period $t \in \mathcal{T}$, the constraint \eqref{eq:constrSpread} implements a copy of the set $\mathcal{S}(w_t,x_{et})$.  If the relaxation $\mathcal{Z}^{\mathrm{2M}}_2$ is used to generate these sets, proxy variables $(\delta_{i,et},\omega_{i,et},\zeta_{i,et})$ are introduced for all $e$ and $t$.  However, for a fixed period $t$, the same wind vector variable $w_t$ appears in the set $\mathcal{S}(w_t,x_{et})$ for all elements $e$.  As the proxy variables $\omega_{i,et}$ represent the squared components of $w_t$, we can also require that these variables take the same value across copies of $\mathcal{S}(w_t,x_{et})$ in the same period $t$.  That is, we can enforce $\omega_{i,e_1t} = \omega_{i,e_2t}$ for any $\{e_1,e_2\} \subseteq \mathcal{E}$, which tightens this relaxation across elements.

\subsubsection{Rotated McCormick Inequalities}

Another approach to handle the inner product is to relax the ball constraints to box constraints, so that the problem becomes separable. Then, we can apply the upper-bounding inequalities of the McCormick envelope \citep{mccormick1976computability}:
\begin{equation}
    \label{eq:McCormick}
    \begin{aligned}
        \mathcal{Z}^{\mathrm{IP}}_n & \subseteq \left \{(d,w,z) \,:\, \begin{array}{rl}
        z & \leq \ip{d}{w}\\
        d & \in \left [ -\overline{R},\overline{R} \right ] ^n\\
        w & \in [-\varepsilon,\varepsilon]^n
    \end{array} \right \}\\
    & \subseteq \left \{ (d,w,z) \,:\, \begin{array}{rlr}
        z & \leq \sum_{i=1}^n z_i &\\
        z_i & \leq \varepsilon (d_i + \overline{R}) - \overline{R}w_i & \forall i \in \Brack{n}\\
        z_i & \leq \overline{R} (w_i + \varepsilon) - \varepsilon d_i & \forall i \in \Brack{n}
    \end{array} \right \}.
    \end{aligned}
\end{equation}
Although valid, this approach ignores the interaction between components of $d$ and $w$ in the ball constraints.  To incorporate this interaction, we observe that $\mathcal{Z}^{\mathrm{IP}}_n$ is invariant to a rotation of the Cartesian axes.  
For any rotation, we can then add the upper-bounding McCormick inequalities and tighten the relaxation.

As in \eqref{abstractOpt}, we fix the dimension to $n = 2$.  Let $M_\theta \in \mathbb{R}^{2 \times 2}$ be the rotation matrix for angle $\theta \in \mathbb{R}$:
\begin{equation*}
    M_\theta := \begin{bmatrix}
        \cos(\theta) & -\sin(\theta)\\
        \sin(\theta) & \cos(\theta)
    \end{bmatrix}.
\end{equation*}
As ${M_\theta} \tp M_\theta = I_2$, for any $(d,w,z) \in \mathcal{Z}^{\mathrm{IP}}_2$, we have that
\begin{equation*}
    \begin{aligned}
        z & \leq \ip{M_\theta d}{M_\theta w},\\
        M_\theta d & \in \mathcal{B}_{\overline{R}}(0_2),\text{ and } M_\theta w \in \mathcal{B}_\varepsilon(0_2);
    \end{aligned}
\end{equation*}
that is, a rotation of the coordinate system by $\theta$ yields the same set.  Therefore, imposing the McCormick inequalities after rotation gives valid inequalities for $\mathcal{Z}^{\mathrm{IP}}_2$. A McCormick upper bound in the rotated space is 
\begin{equation}
    \label{eq:MCUBinRotSpace}
    z \leq \overline{R} \ip{\mathbf{e}_2}{M_\theta w} - \varepsilon \ip{\mathbf{e}_2}{M_\theta d} + 2\varepsilon\overline{R}.
\end{equation}
Note that we selected one of the two McCormick inequalities for each component of the inner product; any combination of the inequalities will yield an equivalent result.

For all $\theta$, inequality~\eqref{eq:MCUBinRotSpace} is valid for $\mathcal{Z}^{\mathrm{IP}}_2$.  We write a robust constraint over $\theta$ and replace $M_\theta$ with its definition:
\begin{equation*}
    z \leq 2\varepsilon\overline{R} + \min_\theta\ \begin{array}{l}
        \left (\overline{R}(w_1 + w_2) - \varepsilon(d_1 + d_2) \right ) \cos(\theta)\\ + \left (\overline{R} (w_1 - w_2) - \varepsilon (d_1 - d_2) \right ) \sin(\theta).
    \end{array}
\end{equation*}
Under the change of variables $c := \cos(\theta)$ and $s := \sin(\theta)$ with constraint $c^2 + s^2 = 1$, the minimization is of a linear function over the surface of a ball, which is solvable in closed form.  This yields the inequality
\begin{equation}
    \label{eq:RotMCLoose}
    \norm{\begin{bmatrix}
        \overline{R} (w_1 + w_2) - \varepsilon (d_1 + d_2)\\
        \overline{R} (w_1 - w_2) - \varepsilon (d_1 - d_2)
    \end{bmatrix}}_2 \leq 2 \varepsilon \overline{R} - z.
\end{equation}
This expression is a convex SOC constraint that implies the McCormick inequalities under every axis rotation. However, squaring both sides reveals an extraneous $z^2$ term on the right-hand side.  This term can be canceled, tightening the constraint while maintaining its validity and convexity:
\begin{equation}
    \label{eq:RotMC}
    \norm{\begin{bmatrix}
        \overline{R}(w_1 + w_2) - \varepsilon(d_1 + d_2)\\
        \overline{R} (w_1 - w_2) - \varepsilon (d_1 - d_2)\\
        z
    \end{bmatrix}}_2 \leq 2 \varepsilon \overline{R} - z.
\end{equation}
We name the tightened constraint~\eqref{eq:RotMC} the \emph{rotated McCormick inequality}.  Let $\mathcal{Z}^{\mathrm{RMC}} := \{(d,w,z) \,:\, \eqref{eq:RotMC}\}$.  Theorem~\ref{thm:rotMCValid} gives a direct proof that these inequalities are valid for $\mathcal{Z}^{\mathrm{IP}}_2$.

\begin{theorem}
    \label{thm:rotMCValid}
    For any $(d,w,z) \in \mathcal{Z}^{\mathrm{IP}}_2$, the inequalities \eqref{eq:RotMCLoose} and \eqref{eq:RotMC} are satisfied.  Thus, $\mathcal{Z}^\mathrm{RMC} \supseteq \mathcal{Z}^{\mathrm{IP}}_2$.
\end{theorem}

\begin{proof}
For $(d,w,z) \in \mathcal{Z}^{\mathrm{IP}}_2$, observe that
\begin{equation*}
    \begin{aligned}
        & \norm{\begin{bmatrix}
            \overline{R}(w_1 + w_2) - \varepsilon(d_1 + d_2)\\
            \overline{R} (w_1 - w_2) - \varepsilon (d_1 - d_2)
        \end{bmatrix}}^2_2\\
        =\ & 2 \overline{R}^2 \norm{w}_2^2 + 2 \varepsilon^2 \norm{d}_2^2 - 4 \varepsilon \overline{R} \ip{d}{w}\\
        \leq\ & 4 \varepsilon^2 \overline{R}^2 - 4 \varepsilon \overline{R} z \ =\ (2 \varepsilon \overline{R} - z)^2 - z^2\\
        \leq\ & \left (2 \varepsilon \overline{R} - z \right )^2,
    \end{aligned}
\end{equation*}
where the first inequality follows from the definition of $\mathcal{Z}^{\mathrm{IP}}_2$.  The validity of \eqref{eq:RotMC} follows from the first inequality and that of \eqref{eq:RotMCLoose} follows from the second inequality.  Note that $z \leq \ip{d}{w} \leq \varepsilon \overline{R}$ gives the nonnegativity of the right hand sides.
\end{proof}

\subsubsection{Inner Product Spread Set Formulation}

We define a relaxation of the Rothermel spread set
\begin{equation}
    \mathcal{S}^{\langle \rangle}(w,x) := \left \{d + x \,:\, \begin{array}{l}
        \norm{d}_2 \leq V(x) \left (1 + C \left (\frac{z }{\norm{d}_2} \right ) ^B \right )\\
        (d,w,z) \in \overline{\mathcal{Z}}
    \end{array} \right \},
\end{equation}
where $\overline{\mathcal{Z}}$ is a convex relaxation of $\mathcal{Z}^{\mathrm{IP}}_2$.  For $B \in [0,1]$, the \emph{inner product spread set} $\mathcal{S}^{\langle \rangle}(w,x)$ is convex even when $w$ is treated as a variable and can be reformulated using power cones as in~\eqref{eq:conicRothermel}.  Note that, if $\overline{w}_t$ is not zero, the inner product $z$ also contains the term $\ip{\overline{w}_t}{d}$; we suppress this dependence of $\mathcal{S}^{\langle \rangle}(w,x)$ on $\overline{w}_t$.  Figure~\ref{fig:IPSpreadRegions} shows the the geometry of $\mathcal{S}^{\langle \rangle}(w,0_2)$ for different values of $w$ and choices of the set $\overline{\mathcal{Z}}$.  The parameters $(B,C,V)$ are calibrated as described in Section~\ref{sec:data}, with $B = 0.9093$.  In this figure, CH denotes the convex hull relaxation $\left (\overline{\mathcal{Z}} = \mathcal{Z}^{\mathrm{CH}}_2 \right )$, 2M denotes the second-order minor constraints $\left ( \overline{\mathcal{Z}} = \mathcal{Z}^{\mathrm{2M}}_2 \right )$, RMC denotes the rotated McCormick inequality $\left ( \overline{\mathcal{Z}} = \mathcal{Z}^{\mathrm{RMC}} \right )$, and MC denotes the standard McCormick inequalities \eqref{eq:McCormick}.  The rotated McCormick often gives a close approximation of the convex hull and is stronger than the other relaxations.  As the second-order minor constraints can be tightened by sharing variables across elements $e \in \mathcal{E}$, we use both these constraints and the rotated McCormick inequalities, and take $\overline{\mathcal{Z}} = \mathcal{Z}^{\mathrm{RMC}} \cap \mathcal{Z}^{\mathrm{2M}}_2$.

%%%%%%%%%%%%%%%%%% Figure %%%%%%%%%%%%%%%%%%
\begin{figure}[tp]
    \centering
    \includegraphics{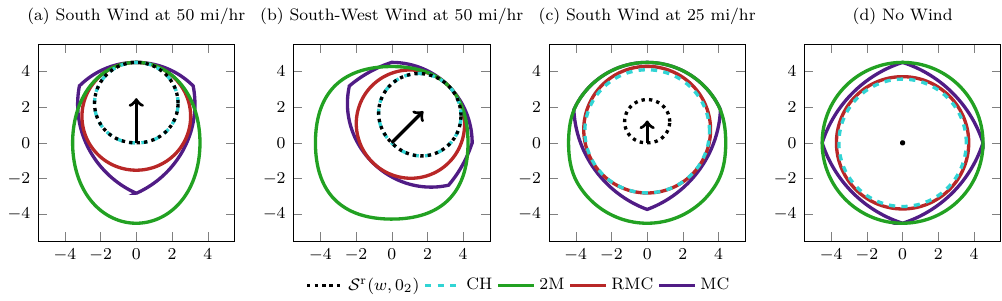}
    \caption{Comparison of spread sets $\mathcal{S}^{\langle \rangle}(w,0_2)$ for each convex relaxation of the inner product.  Wind velocity $w$ is fixed in each plot.  We set $B = 0.9093$, $\overline{w} = 0_2$, and $\varepsilon = 50\ \text{mi/hr}$.  Units are miles; arrows show the wind direction and relative magnitude.}
    \label{fig:IPSpreadRegions}
\end{figure}
%%%%%%%%%%%%%%%%%% Figure %%%%%%%%%%%%%%%%%%
 
\subsection{Ball Spread Set}
\label{subsec:ballrelax}

Instead of relaxing the inner product, we can instead directly relax the set $\mathcal{S}^{\mathrm{r}}(w,x)$.  We propose a principled relaxation of $\mathcal{S}^{\mathrm{r}}(w,x)$ that yields a tractable convex set in the $(w,x)$ space.  We develop this relaxation under the assumption that $B \geq 1$, providing an alternative to the inner product spread set, which requires that $B \leq 1$.  
Lemma~\ref{lemma:ballSum} gives a fundamental property on the Minkowski sum of Euclidean balls.  This property is used to establish the other results in this section.

\begin{lemma}
    \label{lemma:ballSum}
    Let $\{c_{\mathrm{A}},c_{\mathrm{B}}\} \subset \mathbb{R}^n$ and $\{r_{\mathrm{A}},r_{\mathrm{B}}\} \subset \mathbb{R}_\geq$.  Then, $\mathcal{B}_{r_{\mathrm{A}}}(c_{\mathrm{A}}) + \mathcal{B}_{r_{\mathrm{B}}}(c_{\mathrm{B}}) = \mathcal{B}_{r_{\mathrm{A}}+r_{\mathrm{B}}}(c_{\mathrm{A}} + c_\mathrm{B})$.
\end{lemma}

For convenience, we define the functions
\begin{equation}
    \tau(\nu) := V(x) + \frac{C V(x)}{2} \nu^B \  \text{and} \  \sigma(\nu) := \frac{C V(x)}{2} \nu^{B-1}.
\end{equation}
Note that the dependence of these functions on $x$ is suppressed.
We proceed by splitting the spread direction $d$ into two components: $d^{\mathrm{c}}$, which represents the spread due to constant flux, and $d^{\mathrm{w}}$, which represents spread due to wind-assisted flux.  This split gives the relaxed set:
\begin{equation}
    \mathcal{S}^{\angle}(w,x) := \left \{d + x : \begin{array}{l}
        d = d^{\mathrm{c}} + d^{\mathrm{w}}\\
        \norm{d^{\mathrm{c}}}_2 \leq V(x)\\
        \norm{d^{\mathrm{w}}}^2_2 \leq 2\sigma (\norm{w}_2) \ip{d^{\mathrm{w}}}{w}
    \end{array} \right \}.
\end{equation}
If $d^{\mathrm{c}}$ and $d^{\mathrm{w}}$ are parallel and $B = 1$, this set is exactly $\mathcal{S}^\mathrm{r}(w,x)$.  The relaxation is due to the possibility that the constant and wind-assisted spread are in different directions (hence the superscript $\angle$), and the replacement of the exponent $B$ with $1$.  Proposition~\ref{prop:relaxedSpread} establishes that this set relaxes the true spread set and characterizes it exactly as a Euclidean ball.  

\begin{proposition}
    \label{prop:relaxedSpread}
    Fix some $\{w,x\} \subset \mathbb{R}^2$.  Let $B \geq 1$ and $(C,V(x)) \geq 0$. Then, $\mathcal{S}^{\mathrm{r}}(w,x) \subseteq \mathcal{S}^{\angle}(w,x)$.  Further, 
    \begin{equation*}
        \mathcal{S}^{\angle}(w,x) = \mathcal{B}_{\tau(\norm{w}_2)} (\sigma(\norm{w}_2) w + x).
    \end{equation*}
\end{proposition}

\begin{proof}
    Let $x' \in \mathcal{S}^{\mathrm{r}}(w,x)$ and $d := x' - x$, so $\norm{d}_2 \leq V(x) \left (1 + C \left (\frac{\langle d,w \rangle }{\norm{d}_2} \right ) ^B \right )$.  If $\norm{d}_2 \leq V(x)$, let $d^{\mathrm{c}} = d$ and $d^{\mathrm{w}} = 0_2$, certifying that $x' \in \mathcal{S}^{\angle}(w,x)$.  Otherwise, let $d^{\mathrm{c}} = V(x) d/\norm{d}_2$ and $d^{\mathrm{w}} = d - d^{\mathrm{c}}$.
    Then, 
    \begin{equation*}
        \begin{aligned}
            \norm{d^{\mathrm{w}}}_2 & = \norm{d}_2 - V(x)\\
            & \leq C V(x) \left (\frac{\ip{d}{w}}{\norm{d}_2} \right )^B\\
            & = C V(x) \norm{w}_2^B \left (\frac{\ip{d}{w}}{\norm{d}_2 \norm{w}_2} \right )^B\\
            & \leq C V(x) \norm{w}_2^{B} \left (\frac{\ip{d}{w}}{\norm{d}_2 \norm{w}_2} \right ),
        \end{aligned}
    \end{equation*}
    where the final inequality follows as $B \geq 1$ and $\ip{d}{w} / (\norm{d}_2 \norm{w}_2) \leq 1$.  Recall that $\ip{d}{w} \geq 0$ by our definition of the power constraint.  Now, as $d^{\mathrm{w}}$ is a positive scalar multiple of $d$, we have that $\ip{d}{w} / \norm{d}_2 = \ip{d^{\mathrm{w}}}{w} / \norm{d^{\mathrm{w}}}_2$, and again $x' \in \mathcal{S}^{\angle}(w,x)$.  Therefore, $\mathcal{S}^{\mathrm{r}}(w,x) \subseteq \mathcal{S}^{\angle}(w,x)$.

    To establish the second result, observe that $\mathcal{B}_{\norm{c}_2}(c) = \{d \,:\, \norm{d}_2^2 \leq 2\ip{c}{d} \}$.  Then, the final two constraints of $\mathcal{S}^\angle(w,x)$ are equivalent to $d^{\mathrm{c}} \in \mathcal{B}_{V(x)}(0_2)$ and $d^{\mathrm{w}} \in \mathcal{B}_{\tau(\norm{w}_2) - V(x)}(\sigma(\norm{w}_2) w)$.  The set $\mathcal{S}^{\angle}(w,x)$ is the Minkowski sum of these two balls and the point $x$.  By Lemma~\ref{lemma:ballSum}, we have the desired result.
\end{proof}

By the characterization in Proposition~\ref{prop:relaxedSpread}, it is clear that $\mathcal{S}^{\angle}(w,x)$ is convex for fixed $w$.  However, when we optimize simultaneously over $w$ and $\mathcal{S}^{\angle}(w,x)$, the dependence of the radius of the ball on $\norm{w}_2$ introduces nonconvexity.  To tackle this nonconvexity, we further relax $\mathcal{S}^{\angle}(w,x)$.  The proposed relaxation is exactly the convex hull of $\mathcal{S}^{\angle}(w,x)$ in the $(w,x)$ space if the nominal wind prediction $\overline{w} = 0_2$.  Proposition~\ref{prop:ballCH} defines the relaxed set $\mathcal{S}^{\circ}(w,x)$ and gives the convex hull property.  The set $\mathcal{S}^{\circ}(w,x)$ is also a Euclidean ball, with a slightly modified radius and center; for this reason, we call this set the $\emph{ball spread set}$.

\begin{proposition}
    \label{prop:ballCH}
    Fix some $\{x,\overline{w}\} \subset \mathbb{R}^2$.  Let $B \geq 1$ and $(C,V(x)) \geq 0$.  Then,
    \begin{equation*}
        \begin{aligned}
            \mathcal{A}_1 & :=\ \conv \{(w,x') \,:\, w \in \mathcal{B}_\varepsilon(\overline{w}),\, x' \in \mathcal{S}^{\angle}(w,x)\} \\
            & \ \ \subseteq \{(w,x') \,:\, w \in \mathcal{B}_\varepsilon(\overline{w}),\, x' \in \mathcal{S}^{\circ}(w,x)\} := \mathcal{A}_2,
        \end{aligned}
    \end{equation*}
    where 
    $\mathcal{S}^{\circ}(w,x) := \mathcal{B}_{\tau(\norm{\overline{w}}_2+\varepsilon)}(\sigma(\norm{\overline{w}}_2 + \varepsilon) w + x)$.
    If $\overline{w} = 0_2$, then $\mathcal{A}_1 = \mathcal{A}_2$.
\end{proposition}

\begin{proof}
Let $w \in \mathcal{B}_\varepsilon(\overline{w})$ and $(d,d^{\mathrm{c}},d^{\mathrm{w}})$ satisfy the constraints of $\mathcal{S}^{\angle}(w,x)$. Then,
\begin{equation*}
    \begin{aligned}
        \norm{d^{\mathrm{w}}}_2^2 & \leq C V(x) \norm{w}_2^{B-1} \ip{d^{\mathrm{w}}}{w}\\
        & \leq C V(x) (\norm{\overline{w}}_2 + \varepsilon)^{B-1} \ip{d^{\mathrm{w}}}{w},
    \end{aligned}
\end{equation*}
as $B - 1 \geq 0$.
Then, following the logic in the proof of Proposition~\ref{prop:relaxedSpread}, $d^{\mathrm{w}} \in \mathcal{B}_{\tau(\norm{\overline{w}}_2 + \varepsilon) - V(x)}(\sigma(\norm{\overline{w}}_2 + \varepsilon)w)$ and $d^{\mathrm{c}} \in \mathcal{B}_{V(x)}(0_2)$.  Then, $d + x$ is in the sum of these Euclidean balls and the point $x$, which is exactly $\mathcal{S}^{\circ}(w,x)$ by Lemma~\ref{lemma:ballSum}.  Therefore, $\mathcal{S}^{\angle}(w,x) \subseteq \mathcal{S}^{\circ}(w,x)$.  As $\mathcal{A}_2$ is a convex set, we have $\mathcal{A}_1 \subseteq \mathcal{A}_2$.

Next, assume $\overline{w} = 0_2$.  If $\varepsilon = 0$, it holds that $\mathcal{S}^{\angle}(\overline{w},x) = \mathcal{S}^{\circ}(\overline{w},x)$, and $\mathcal{A}_1 = \mathcal{A}_2$ trivially.  Otherwise, let $w \in \mathcal{B}_\varepsilon(0_2)$ and
\begin{equation*}
    \begin{aligned}
        w_\mathrm{A} = \frac{\varepsilon w}{\norm{w}_2}, && && w_\mathrm{B} = -\frac{\varepsilon w}{\norm{w}_2},\\
        \lambda_{\mathrm{A}} = \frac{\varepsilon + \norm{w}_2}{2\varepsilon}, && \text{  and  } && \lambda_{\mathrm{B}} = \frac{\varepsilon - \norm{w}_2}{2\varepsilon}.
    \end{aligned}
\end{equation*}
Note that $\lambda \geq 0$ and $\lambda_{\mathrm{A}} + \lambda_{\mathrm{B}}=1$, so $\lambda$ are convex multipliers.  Additionally, $\lambda_{\mathrm{A}} w_{\mathrm{A}} + \lambda_{\mathrm{B}} w_{\mathrm{B}} = w$ and $\{w_{\mathrm{A}},w_{\mathrm{B}}\} \subset \mathcal{B}_{\varepsilon}(0_2)$.  By Lemma~\ref{lemma:ballSum} and Proposition~\ref{prop:relaxedSpread}, 
\begin{equation*}
    \begin{aligned}
        & \lambda_{\mathrm{A}} \mathcal{S}^{\angle}(w_{\mathrm{A}},x) + \lambda_{\mathrm{B}} \mathcal{S}^{\angle}(w_{\mathrm{B}},x)\\
        = \ & \mathcal{B}_{\lambda_{\mathrm{A}} \tau(\varepsilon)} (\lambda_{\mathrm{A}} \sigma (\varepsilon) w_{\mathrm{A}} + \lambda_{\mathrm{A}} x)\\
        & + \mathcal{B}_{\lambda_{\mathrm{B}} \tau(\varepsilon)} (\lambda_{\mathrm{B}} \sigma (\varepsilon) w_{\mathrm{B}} + \lambda_{\mathrm{B}} x)\\
        = \ & \mathcal{B}_{\tau(\varepsilon)}(\sigma(\varepsilon) w + x) = \mathcal{S}^{\circ}(w,x).
    \end{aligned}
\end{equation*}
Then, $\{w\} \times \mathcal{S}^{\circ}(w,x) \subseteq \mathcal{A}_1$.  As this holds for all $w \in \mathcal{B}_{\varepsilon}(0_2)$, we have that $\mathcal{A}_2 \subseteq \mathcal{A}_1$ when $\overline{w} = 0_2$.
\end{proof}

Figure~\ref{fig:ballSpreadRegions} depicts the geometry of the initial relaxed spread set $\mathcal{S}^{\angle}(w,0_2)$ and of the ball spread set $\mathcal{S}^{\circ}(w,0_2)$.  The parameters $(C,V)$ are calibrated as in Section~\ref{sec:data} and the exponent $B$ is rounded to $1$.  In this figure, the maximum rate of spread is larger than that in Figure~\ref{fig:IPSpreadRegions} as it is now computed with $B = 1$, instead of $B < 1$.  We find that the relaxed set $\mathcal{S}^{\angle}(w,0_2)$ corresponds almost exactly with the true set, while the ball spread set introduces some error when the wind speed is not at its upper bound. 

%%%%%%%%%%%%%%%%%% Figure %%%%%%%%%%%%%%%%%%
\begin{figure}[tp]
    \centering
    \includegraphics{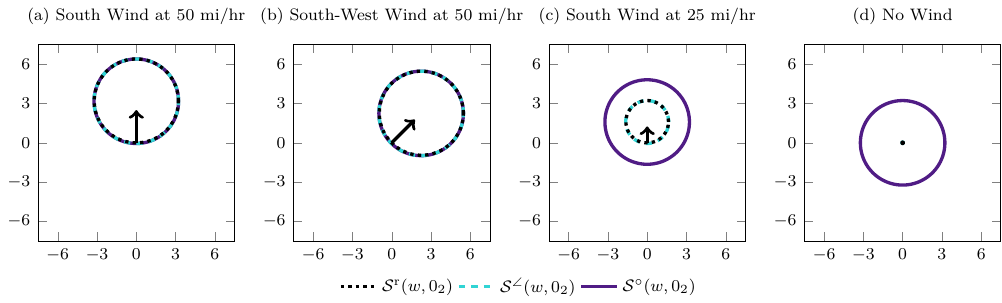}
    \caption{Comparison of the Rothermel, relaxed, and ball spread sets.  Wind velocity $w$ is fixed in each plot.  We set $B = 1$, $\overline{w} = 0_2$, and $\varepsilon = 50\ \text{mi/hr}$.  Units are miles; arrows show the wind direction and relative magnitude.}
    \label{fig:ballSpreadRegions}
\end{figure}
%%%%%%%%%%%%%%%%%% Figure %%%%%%%%%%%%%%%%%%

\subsection{Choosing a Relaxation}
\label{subsec:chooserelax}

We have introduced two distinct convex relaxations of the Rothermel spread set.  Depending on the context, it may be preferable to use one relaxation over the other.  Here, we provide some guidelines on which relaxation may be more accurate; these guidelines are summarized in Table~\ref{tbl:relaxationComparison}.

%%%%%%%%%%%%%%%%%% Figure %%%%%%%%%%%%%%%%%%
\begin{table}[tp]
    \centering
    \begin{tabular}{@{}cccc@{}}
        \toprule
         & $B \ll 1$ & $ B \sim 1$ & $B \gg 1$ \\
         \midrule
         $C \overline{W}^B \sim 1$ & $\mathcal{S}^{\langle \rangle}$ & $\mathcal{S}^{\langle \rangle}$ & $\mathcal{S}^{\circ}$ \\
         $C \overline{W}^B \not \sim 1 $ &  $\mathcal{S}^{\langle \rangle}$ & $\mathcal{S}^{\circ}$ & $\mathcal{S}^{\circ}$ \\
        \bottomrule
    \end{tabular}
    \caption{Preferred relaxation of $\mathcal{S}^{\mathrm{r}}$ by relationship between the parameters $(B,C,\overline{W})$.}
    \label{tbl:relaxationComparison}
\end{table}
%%%%%%%%%%%%%%%%%% Figure %%%%%%%%%%%%%%%%%%

The convexity of the inner product spread set $\mathcal{S}^{\langle \rangle}$ relies on the condition $B \leq 1$.  If $B > 1$, it must be rounded down to $1$ to use this relaxation.  On the other hand, the ball spread set $\mathcal{S}^{\circ}$ is designed under the assumption that $B \geq 1$, and $B$ must be rounded up to $1$ if $B < 1$.  Therefore, if $B \gg 1$, it is preferable to use the ball spread set, and if $B \ll 1$, the inner product spread set is preferable.

Additionally, the ball spread set omits the relationship between the base direction of spread and the wind-assisted direction of spread.  If the magnitude of one of these components dominates the magnitude of the other, the relaxation that generates the ball spread set is relatively tight.  If the magnitudes are on the same order, this relaxation is looser.  That is, if $1 \sim C\overline{W}^B$, where $\overline{W}$ is the maximum wind speed, then the inner product spread set is preferable, and otherwise the ball spread set is preferable.

%%Explicit Mixed-Integer Conic Formulations
\section{Explicit Mixed-Integer Conic Formulations}
\label{sec:micpFormulation}

This section details a formulation that implements the logic of problem~\eqref{abstractOpt}.  In Section~\ref{subsec:micpFormulation}, we describe a mixed-binary description of the outage indicator definition \eqref{eq:constrOutage} and of the region-dependent rate of spread functions $V(x)$.  In Section~\ref{subsec:SOCPower}, we describe a second-order cone approximation of the power cone constraints to facilitate implementation of the set $\mathcal{S}^{\langle \rangle}(w,x)$.

\subsection{Mixed-Integer Conic Description of the Adversarial Model}
\label{subsec:micpFormulation}

In problem~\eqref{abstractOpt}, the constraints \eqref{eq:constrMonotonic}-\eqref{eq:constrWind} can easily be written as linear or quadratic constraints, and \eqref{eq:constrBin} enforces the binary domain.  It remains to formulate the objective \eqref{eq:obj} and the constraints \eqref{eq:constrSpread}-\eqref{eq:constrOutage}.  To permit reformulation, we make the following assumption.
\begin{assumption}  
\label{assump:data}
The following properties of the model data hold:
    \begin{enumerate}[label=\alph*.]
        \item The regions $\mathcal{P}_r$ and geometries $\mathcal{L}_e$ are nonempty polytopes for all $r \in \mathcal{R}$ and $e \in \mathcal{E}$;\label{assump:boundedRegions}
        \item The union $\bigcup_{r \in \mathcal{R}} \mathcal{P}_r$ is a box and contains $\mathcal{L}_e$ for all $e \in \mathcal{E}$; and\label{assump:regionCoverage}
        \item There is an optimal solution to \eqref{abstractOpt} with $o_{eT} = 1$ for all $e \in \mathcal{E}$.\label{assump:universalOutage}
    \end{enumerate}
\end{assumption}
 
Assumption~\hyperref[assump:boundedRegions]{\ref*{assump:data}\ref*{assump:boundedRegions}} requires that both the rate of spread regions and grid element geometries are nonempty, polyhedral, and bounded, which holds by construction (see Section~\ref{sec:data}).  Assumption~\hyperref[assump:regionCoverage]{\ref*{assump:data}\ref*{assump:regionCoverage}} establishes that the rate of spread regions cover all element geometries and all areas between elements.  Assumption~\hyperref[assump:universalOutage]{\ref*{assump:data}\ref*{assump:universalOutage}} requires that some optimal solution has all elements outaged by the final period $T$.  If the objective values $c_t(\cdot)$ are nondecreasing, this property will be satisfied if $T$ is chosen to be sufficiently large.

To formulate the objective \eqref{eq:obj}, we introduce binary variables $\overline{o}_{\mathcal{E}'t}$ for every subset $\mathcal{E}' \subseteq \mathcal{E}$.  These variables reformulate the product terms in the objective, indicating whether the subset $\mathcal{E}'$ is outaged.  The variables $\overline{o}$ are subject to the following constraints:
\begin{subequations}
    \label{eq:constrOutageComb}
    \begin{align}
        & \overline{o}_{\mathcal{E}'t} \leq o_{et} \quad && \forall e \in \mathcal{E}',\ \mathcal{E}' \subseteq \mathcal{E},\ t \in \mathcal{T}, \label{eq:constrOutageCombDef}\\
        & \overline{o}_{\mathcal{E}'t} \leq 1 - o_{et} \quad && \forall e \in \mathcal{E} \setminus \mathcal{E}',\ \mathcal{E}' \subseteq \mathcal{E},\ t \in \mathcal{T}, \label{eq:constrOutageComplCombDef}\\
        & \sum_{\mathcal{E}' \subseteq \mathcal{E}} \overline{o}_{\mathcal{E}'t} \leq 1 \quad && \forall t \in \mathcal{T}, \label{eq:constrOutageCombTotal}\\
        & \overline{o}_{\mathcal{E}'t} \in \{0,1\} \quad && \forall \mathcal{E}' \subseteq \mathcal{E},\ t \in \mathcal{T}. \label{eq:constrOutageCombBin}
    \end{align}
\end{subequations}
Constraints~\eqref{eq:constrOutageCombDef}~and~\eqref{eq:constrOutageComplCombDef} ensure that the variable $\overline{o}_{\mathcal{E}'t}$ is $1$ only if exactly the elements $e \in \mathcal{E}'$ are outaged, and constraint~\eqref{eq:constrOutageCombTotal} requires that only one combination of elements is outaged in each period.
In the augmented variable space, the objective \eqref{eq:obj} can be written as a linear function:
\begin{equation*}
    \sum_{\mathcal{E}' \subseteq \mathcal{E}}\ \sum_{t \in \mathcal{T}} c_t(\mathcal{E}') \ \overline{o}_{\mathcal{E}'t}.
\end{equation*}
For the minimum time-to-outage objective, note that only variables $\{\overline{o}_{\mathcal{E}t}\}_{t \in \mathcal{T}}$ need to be introduced.  

Next, we address the spread constraint~\eqref{eq:constrSpread}.  The coefficient $V(x)$ requires indicators of the region $r \in \mathcal{R}$ containing $x$.
By Assumption~\hyperref[assump:boundedRegions]{\ref*{assump:data}\ref*{assump:boundedRegions}}, each region is a polytope $\mathcal{P}_r$; let the H-representation of these polytopes be given by $\mathcal{P}_r = \{y \,:\, A_r y \leq b_r\}$.  The Balas formulation for the union of polytopes \citep{jeroslow1984modelling,balas1985disjunctive} gives a well-behaved representation for the desired indicators.  This representation introduces region indicator variables $p_{ert}$ and proxy variables $x_{ert} \in \mathbb{R}^2$, constrained by the following logic:
\begin{subequations}
    \begin{align}
        & A_r x_{ert} \leq b_r p_{ert} \quad && \forall e \in \mathcal{E},\ r \in \mathcal{R},\ t \in \mathcal{T}, \label{eq:constrRegionBounds}\\
        & \sum_{r \in \mathcal{R}} p_{ert} = 1 \quad && \forall e \in \mathcal{E},\ t \in \mathcal{T}, \label{eq:constrRegionCount}\\
        & \sum_{r \in \mathcal{R}} x_{ert} = x_{et} \quad && \forall e \in \mathcal{E},\ t \in \mathcal{T} \label{eq:constrRegionTotal}\\
        & p_{ert} \in \{0,1\} \quad && \forall e \in \mathcal{E},\ r \in \mathcal{R},\ t \in \mathcal{T}. \label{eq:constrRegionBin}
    \end{align}
\end{subequations}
As the indicators $p_{ert}$ are binary and $\mathcal{P}_r$ are bounded, constraint~\eqref{eq:constrRegionBounds} requires either that $p_{ert} = 0$ and $x_{ert} = 0_2$, or that $p_{ert} = 1$ and $x_{ert} \in \mathcal{P}_r$.  Constraint~\eqref{eq:constrRegionCount} enforces that exactly one region indicator is active, and constraint~\eqref{eq:constrRegionTotal} then requires that $x_{et}$ is in the active region.  Under Assumption~\hyperref[assump:regionCoverage]{\ref*{assump:data}\ref*{assump:regionCoverage}}, this formulation does not unduly constrain the feasible values of the fire locations $x_{et}$.

We apply a similar idea as the Balas formulation to implement the dependence of the spread rate on the active region.  Let variables $d_{ert} \in \mathbb{R}^2$ be the proxy variables for the spread direction $d$.  Then, we enforce
\begin{subequations}
    \begin{align}
        & \norm{d_{ert}}_{\infty} \leq \mu_r \overline{R}  p_{ert} \quad && \forall e \in \mathcal{E},\ r \in \mathcal{R},\ t \in \mathcal{T},\label{eq:constrRegionSpreadBounds}\\
        & \sum_{r \in \mathcal{R}} d_{ert} = x_{e,t+1} - x_{et} \quad && \forall e \in \mathcal{E},\ t \in \mathcal{T}.\label{eq:constrRegionSpreadTotal}
    \end{align}
\end{subequations}
Constraint~\eqref{eq:constrRegionSpreadBounds} requires that the spread contribution from a region is zero if the corresponding indicator is zero, and otherwise enforces a componentwise bound on the magnitude of the spread $d_{ert}$, where the maximum spread distance $\overline{R}$ is reduced by the multiplier $\mu_r$ for region $r$.  Constraint~\eqref{eq:constrRegionSpreadTotal} enforces that the change in fire location is exactly the spread contribution from the active region.

For the ball and inner product spread sets, we impose constraints that enforce the spread behavior exactly by region.  For the ball spread set $\mathcal{S}^{\circ}(w,x)$, we require
\begin{equation}
    \begin{aligned}
        & \norm{\frac{CV}{2}(\norm{\overline{w}_t}_2 + \varepsilon)^{(B - 1)} w_t - \sum_{r \in \mathcal{R}} \frac{1}{\mu_r} d_{ert}}_2\\
        & \quad \leq V + \frac{CV}{2} (\norm{\overline{w}_t}_2 + \varepsilon)^B \quad \forall e \in \mathcal{E},\ t \in \mathcal{T}.
    \end{aligned}
\end{equation}
As exactly one of $\{d_{ert}\}_{r \in \mathcal{R}}$ is nonzero, multiplying this constraint by $\mu_r$ for the active region $r$ yields the description of $\mathcal{S}^{\circ}(w_t,x_{et})$.  This constraint is a quadratic constraint in the variables $(w,d)$.  Otherwise, for the inner product spread set $\mathcal{S}^{\langle \rangle}(w,x)$, we introduce variables $(\gamma_{1et},\gamma_{2et},\gamma_{3et},z_{et})$ and impose
\begin{subequations}
    \begin{align}
        & \norm{d_{ert}}_2 \leq (\mu_r) ^{\frac{1}{1+B}}\gamma_{1et} \quad && \forall e \in \mathcal{E},\ r \in \mathcal{R},\ t \in \mathcal{T}, \label{eq:constrIPNorm}\\
        & \gamma_{1et} \leq (\gamma_{2et} + \gamma_{3et})^{\frac{1}{1+B}} \quad && \forall e \in \mathcal{E},\ t \in \mathcal{T},\\
        & \gamma_{2et} \leq V \cdot (\gamma_{1et})^B \quad && \forall e \in \mathcal{E},\ t \in \mathcal{T},\\
        & \gamma_{3et} \leq CV \cdot (\ip{\overline{w}_t}{d_{et}} + z_{et})^B \quad && \forall e \in \mathcal{E},\ t \in \mathcal{T},\\
        & (d_{et}, w_t, z_{et}) \in \overline{\mathcal{Z}} \quad && \forall e \in \mathcal{E},\ t \in \mathcal{T}.
    \end{align}
\end{subequations}
These constraints give the power conic reformulation of the inner product spread set described in Proposition~\ref{prop:conicRothermel}.  Constraint~\eqref{eq:constrIPNorm} handles the rescaling of the coefficient $V(x)$ for the active region, so that the variables $\gamma$ and other constraints do not need to be duplicated across regions.  Note that, although an explicit formulation of $\overline{\mathcal{Z}}$ is not provided here, the implicit variables $\omega_{i}$ will be shared across indices $e \in \mathcal{E}$, as the wind vector variable $w_t$ remains the same as this index changes.  This approach helps to tighten the relaxation $\mathcal{Z}^{\mathrm{2M}}_2$ as described in Section~\ref{subsubsec:SOMIP}.

Last, we describe the outage indicator definitions \eqref{eq:constrOutage}.  Let $\mathcal{L}_e = \conv ( \{\ell_{ej}\}_{j=1}^J)$ be a polytope in V-representation.  If element $e$ is a power line, $J = 2$ and $\ell_{ej} \in \mathbb{R}^2$ give the start and end points of the line.  If $e$ is a bus, $J = 1$ and $\ell_{ej}$ gives its location.  This construction also applies to more complicated polyhedral geometries.  First, we require that the fire location for element $e$ reaches the set $\mathcal{L}_e$ in the final period $T$.  With variables $\lambda_{ej}$, this requirement is written as 
\begin{subequations}
    \label{eq:constrfinalGeometryIntersection}
    \begin{align}
        & \sum_{j=1}^J \lambda_{ej} \ell_{ej} = x_{eT} \quad && \forall e \in \mathcal{E},\\
        & \sum_{j=1}^J \lambda_{ej} = 1 \quad && \forall e \in \mathcal{E},\\
        & \lambda_{ej} \geq 0 \quad && \forall e \in \mathcal{E},\ j \in \Brack{J}.
    \end{align}
\end{subequations}
Next, if element $e$ is outaged at period $t$, we do not permit its fire location $x_{et}$ to move between periods.  Then, if an element is outaged at period $t$, by the monotonicity constraint~\eqref{eq:constrMonotonic}, its fire location will not move for the remainder of the horizon and must already lie in the geometry $\mathcal{L}_e$, due to constraints~\eqref{eq:constrfinalGeometryIntersection}.  Under Assumption~\hyperref[assump:universalOutage]{\ref*{assump:data}\ref*{assump:universalOutage}}, there exists an optimal solution with all elements outaged in period $T$, so this reformulation is not restrictive and is equivalent to \eqref{eq:constrOutage}-\eqref{eq:constrMonotonic}.  For the ball spread set $\mathcal{S}^{\circ}(w,x)$, this logic is enforced by 
\begin{equation}
    \label{eq:constrBallSpreadOutage}
    \begin{aligned}
        \norm{d_{ert}}_{\infty} \leq \mu_r \overline{R} o_{et} \quad \forall e \in \mathcal{E},\ r \in \mathcal{R},\ t \in \mathcal{T}.
    \end{aligned}
\end{equation}
For the inner product spread set $\mathcal{S}^{\langle \rangle}(w,x)$, the corresponding constraint is
\begin{equation}
    \begin{aligned}
        \gamma_{1et} \leq \overline{R} o_{et} \quad \forall e \in \mathcal{E},\ t \in \mathcal{T}.
    \end{aligned}
\end{equation}
As $\gamma_{1et}$ already provides provides a bound on the norm of the spread components $d_{ert}$ for all regions, we do not repeat this constraint for every region $r$, as is necessary in constraint~\eqref{eq:constrBallSpreadOutage}.

Under this characterization of the outage indicators, once an element is outaged, the region containing its fire point will not change.  With this intuition, we impose a set of valid constraints:
\begin{equation}
    \label{eq:constrValidRegionChange}
    \begin{aligned}
        |p_{er,t+1} - p_{ert}| \leq 1 - o_{et} \quad \forall e \in \mathcal{E},\ r \in \mathcal{R},\ t \in \mathcal{T}.
    \end{aligned}
\end{equation}

\subsection{Second-Order Conic Representation of the Power Cone}
\label{subsec:SOCPower}

Some commercial optimization solvers do not offer native support of power cone constraints.  To circumvent this issue, we make use of known SOC representations of the power cone with a rational exponent \citep[\eg,][Sec.~3.3]{ben2001lectures}.  Specifically, for integers $(N,D,\rho)$ with $D \leq \Pi :=2^\rho$, the rational power constraint $x \leq y ^ \frac{N}{D}$ is written as a geometric mean:
\begin{equation*}
    x^\Pi \leq x^{\Pi - D} y^N 1^{D-N}.
\end{equation*}
We represent the geometric mean constraint equivalently using a result of \citet[Prop.~5]{morenko2013p}, which introduces $\rho-1$ proxy variables and $\rho$ new SOC constraints.  This representation is more concise than standard reformulations, which require $\Pi-1$ new variables and constraints. 

%%Evaluative Optimal Power Flow Model
\section{Evaluative Optimal Power Flow Model}
\label{sec:opf}

For analysis of wildfire impact on power delivery, we align a security-constrained direct-current (DC) optimal power flow (OPF) model of the transmission grid with the affected area.  This model minimizes the costs associated with power generation and unmet demand, while respecting an approximation of the transmission dynamics.  For an overview of OPF models, see \cite{frank2016introduction}.  This section provides a high-level description of our OPF model; specific mathematical details are provided in Appendix~\ref{app:explictOPF}.

We adopt the operational portion of the capacity expansion model of \cite{musselman2025climate}.  This model optimizes the cost of power generation and storage against the amount of \emph{load shed}, \ie, the quantity of unmet demand.  The model is multi-period over horizon $\mathcal{T}$, where consecutive periods are linked by the amount of stored energy and consumption of hydroelectric power.  Power flows are modeled under the DC approximation.  We let $\mathcal{F}^\mathrm{b}$ denote the feasible set of this base operational model, where $p \in \mathcal{F}^{\mathrm{b}}$ represents a feasible generation, storage, and load shedding decision.  The cost of decision $p$ is given by $O^{\mathrm{b}}(p)$.

To model the impact of element outages, we incorporate contingency costs into the OPF model.  Specifically, let $\mathcal{K}$ denote the set of contingencies, where $(\mathcal{E},t) \in \mathcal{K}$ is the contingency in period $t$ with elements $\mathcal{E}$ outaged.  After the outage occurs, we allow the system to adapt in that period by shedding additional load, reducing generation, and reducing the amount of storage charging or discharging, so that the DC power flows are satisfied.  We treat reductions to the charging of storage in the recourse decision as load shedding, subject to the relevant costs.  The base operational decision gives a minimum amount of load shed and maximum amounts of generation and storage discharge in the contingency; that is, neither the amount of generation nor the amount of load served may increase relative to the base decision after the contingency occurs.  This corrective action gives a continuous model of disconnecting load, generation, and storage from the grid as a response to the contingency.  For contingency $(\mathcal{E},t) \in \mathcal{K}$, let the set $\mathcal{F}^{\mathrm{ctg}}_{\mathcal{E}t}(p)$ denote the feasible set of post-contingency decisions under base operational decision $p$, where the elements in $\mathcal{E}$ are outaged.  These decisions consider post-contingency operations only in the period $t$ when the contingency occurs.  The contingency cost is due to the cost of shedding load, which is given by the function $O^{\mathrm{s}}_{\mathcal{E}t}(p^{\mathrm{ctg}}_{\mathcal{E}t})$ for feasible decision $p^{\mathrm{ctg}}_{\mathcal{E}t} \in \mathcal{F}^{\mathrm{ctg}}_{\mathcal{E}t}(p)$.
The optimal contingency cost as a function of the base operational decision $p$ is given by
\begin{equation}
    O^{\mathrm{ctg}}_{\mathcal{E}t}(p) := \min_{p^{\mathrm{ctg}}_{\mathcal{E}t} \in \mathcal{F}^{\mathrm{ctg}}_{\mathcal{E}t}(p)} \quad O^{\mathrm{s}}_{\mathcal{E}t}(p^{\mathrm{ctg}}_{\mathcal{E}t}).
\end{equation}
As the recourse decision permits decreasing power generation and shedding load, the contingency subproblems are feasible for any $p \in \mathcal{F}^{\mathrm{b}}$.

The security-constrained DC-OPF model is as follows:
\begin{equation}
    \label{DCOPF}
    \begin{aligned}
        \min_{p \in \mathcal{F}^{\mathrm{b}}} \quad \frac{1}{T} O^{\mathrm{b}}(p) + \frac{1}{TK}\sum_{(\mathcal{E},t) \in \mathcal{K}} O^{\mathrm{ctg}}_{\mathcal{E}t}(p),
    \end{aligned}
\end{equation}
where $K$ gives the number of contingencies per period.  The problem~\eqref{DCOPF} is a linear program.  The objective value of this model is the average per-period operational cost plus the average contingency cost.

%%Experimental Setting and Calibration
\section{Experimental Setting and Calibration}
\label{sec:data}

We consider the impact of wildfires on the California power grid, motivated by several recent large and destructive fires.  Specifically, we design three experimental settings, which correspond to the Park fire of 2024, and the Eaton and Palisades fires of January 2025.  These experiments use a model of the California power grid, choose contingency sets near each of these historical fires, and calibrate the wildfire spread model to chaparral, a flammable shrub found throughout California.

\subsection{Power Grid and Contingency Data}
\label{subsec:gridData}
We take the nodal electric transmission network from \cite{musselman2025climate}, which is a modestly aggregated version of the California Test System \citep{taylor2024CATS} with 3,928 buses and 5,581 lines.  This network uses real power grid geography and parameters where available, and otherwise substitutes representative simulated data.  Following the methodology of \cite{musselman2025climate}, we impose realistic temporal data, including predicted load and availability of renewable and hydroelectric resources, for the first 24 hours after the wildfire ignitions ($T = 24$).  The ignition times for each fire are given in Table~\ref{tbl:ignition}.  This approach captures seasonal and hourly variations in the behavior of the power grid around the time of each fire.  Load is predicted using hourly temperature from the ERA5 reanalysis \citep{hersbach2023era5}, and renewable generation availability is taken from the System Advisor Model of the National Renewable Energy Laboratory \citep{blair2018system} with simulated hourly weather data from the Energy Exascale Earth System California Regionally Refined Model \citep{zhang2024leveraging}.  Hydroelectric availability is approximated with data from 2019, taken from the water system model of \cite{yates2024modeling}.  We define the cost of load shedding as $\$10,000 / \text{MWh}$ \citep{miso2025voll}.

%%%%%%%%%%%%%%%%%% Figure %%%%%%%%%%%%%%%%%%
\begin{table}[tp]
    \centering
    \begin{tabular}{@{}crrr@{}}
    \toprule
     & Year & Date & Time \\
    Fire &  &  &  \\
    \midrule
    Park & 2024 & July 24 & 14:00 \\
    Palisades & 2025 & January 7 & 10:00 \\
    Eaton & 2025 & January 7 & 18:00 \\
    \bottomrule
    \end{tabular}
    \caption{Wildfire ignition times (Pacific time).}
    \label{tbl:ignition}
\end{table}
%%%%%%%%%%%%%%%%%% Figure %%%%%%%%%%%%%%%%%%

For each fire, we select the set of contingency elements $\mathcal{E}$ by identifying five lines that are near the area affected by wildfire and that incur a large load shed when outaged.  Specifically, we solve a contingency-free DC-OPF~\eqref{DCOPF} (with $\mathcal{K} = \emptyset$) to an optimal solution $p^*$.  We prescreen for lines that intersect a 5-mile buffer around the final wildfire perimeter, then select for the set $\mathcal{E}$ the five intersecting lines that incur the largest post-contingency load shed under the base solution $p^*$ when outaged; that is, the five lines in the buffer with the largest value of $\sum_{t \in \mathcal{T}} \mathcal{C}^{\mathrm{ctg}}_{\{j\}t}(p^*)$.
Here, $j$ denotes the line index.  As bus outages are equivalent to the outage of all lines incident to the bus, our experiments only consider line contingencies.

Figure~\ref{fig:geography} superimposes the geometry of the transmission network on the state of California.  For each fire, it shows the final perimeter and ignition point, and highlights the five lines that are selected for the contingency set $\mathcal{E}$.  The fire perimeters are compared against the union of the Rothermel spread sets for the first 24 hours of fire spread, using the calibration, nominal wind forecast, and true ignition point described in Section~\ref{subsec:calibration}.  For these sets, the coefficient $V(x)$ takes the average of the region multipliers described in Section~\ref{subsec:regions}, weighted by area.  These spread set unions do not cover the entirety of the fire perimeters for two main reasons: first, the sets summarize 24 hours of fire spread, while the boundaries show the area burned over the entire duration of the fire; and second, spatial variation in the rate of spread due to the coefficient $V(x)$ is averaged out.  In fact, the Park fire boundary at 24~hours after ignition is entirely contained in this spread set union.  The adversarial model expands the area reachable by fire beyond these sets by adding uncertainty in the wind forecast, using relaxations of the Rothermel spread sets, and considering the aforementioned spatial variation.  The Eaton and Palisades fires experienced winds that blew towards urban areas and the California coast, which limited the actual fire spread in those directions.  On the other hand, the Park fire grew significantly in the direction of the wind.  Actual fire spread is also reduced by mitigation efforts, which are not captured in our analysis.

%%%%%%%%%%%%%%%%%% Figure %%%%%%%%%%%%%%%%%%
\begin{figure}[tp]
    \centering
    \includegraphics{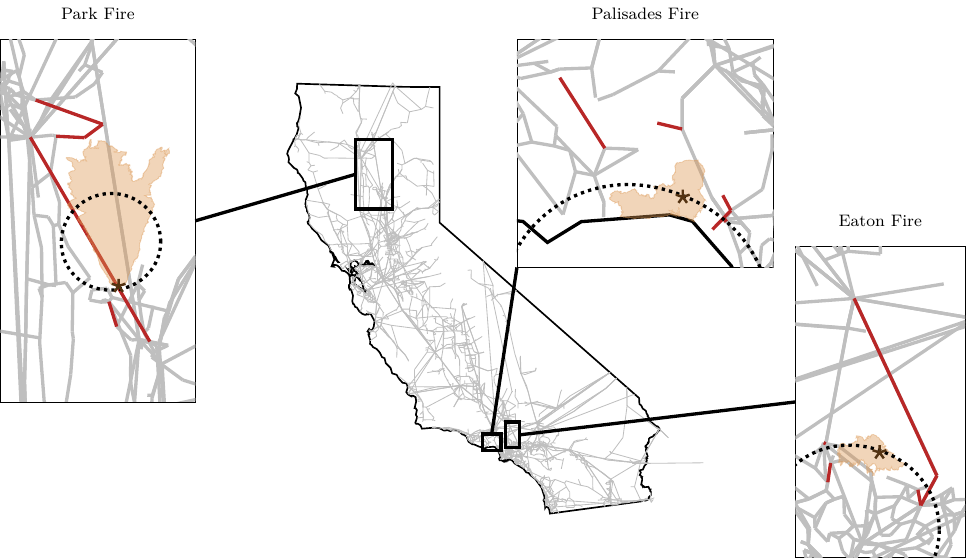}
    \caption{Synthetic power network topology \citep{musselman2025climate} overlaid on the state of California. Final perimeters of the Park, Eaton, and Palisades fires are shown in orange, and lines in the selected contingency set are highlighted in red.  The ignition location of each fire is marked by an asterisk.  The union of the Rothermel spread sets over the first 24 hours of fire spread are shown with dotted lines.}
    \label{fig:geography}
\end{figure}
%%%%%%%%%%%%%%%%%% Figure %%%%%%%%%%%%%%%%%%

\subsection{Fire Spread Model Calibration}
\label{subsec:calibration}

We consider two parameter settings for wildfire uncertainty, which we name the \emph{low flexibility} (LF) and \emph{high flexibility} (HF) settings.  In the low flexibility setting, we fix the ignition point $x_0$ to the true ignition point and use a small wind uncertainty set, centered around a nominal forecast.  In the high flexibility setting, the ignition point remains unfixed, and the wind uncertainty set is large, permitting the wind vectors to point in any direction, with bounded wind speed.

Specifically, in the low flexibility setting, the wind uncertainty set is centered around a nominal forecast $\overline{w}_t$, where the direction $\overline{w}_t / \norm{\overline{w}_t}_2$ is the wind direction at 10 meters above ground level at the end of each hour, and the magnitude $\norm{\overline{w}_t}_2$ takes the value of the maximum 3-second wind gust speed over the hour, in the area affected by wildfire.  These data are collected from the ERA5 reanalysis \citep{hersbach2023era5}.  We use an uncertainty budget of $\varepsilon = 10 \text{ mi/hr}$ and the ignition point $x_0$ is fixed in the optimization problem.  In the high flexibility setting, the nominal forecast $\overline{w}_t = 0_2$ and $\varepsilon$ is $10 \text{ mi/hr}$ larger than the maximum nominal wind speed in the low flexibility setting: $47 \text{ mi/hr}$ for the Park fire, $46 \text{ mi/hr}$ for the Eaton fire, and $50 \text{ mi/hr}$ for the Palisades fire.  The ignition point $x_0$ is not fixed and is treated as a variable in the optimization.

We calibrate the values $(C,B)$, which parametrize the flux multiplier due to wind, to align with data for the California chaparral fuel bed.  We select a packing ratio of 0.016 (dimensionless) and a surface-area-to-volume ratio of $2,500\text{ m}^{-1}$ \citep{zhou2005modeling}.  With the wind speed given in mi/hr, these values imply that $C = 2.5010$ and $B = 0.9093$.
For the inner product spread set, the power cone constraints must have rational exponents with small denominators (see Section~\ref{subsec:SOCPower}).  For the exponent $B$, we use the  approximation $B = 10/11$ with $\rho = 4$, which introduces an approximation error on the order of $10^{-4}$.  We approximate the exponent $1/(1+B) = 0.5238$ by $8/15$, with an approximation error on the order of $10^{-2}$.  

The base rate of spread $V$ is determined from historical Park fire boundaries.  For each direction $d \in \mathbb{R}^2$, we compute the 24-hour flux multiplier using the Rothermel spread set:
\begin{equation*}
    \sum_{t \in \mathcal{T}} \left (1 + C \cdot \left (\frac{\ip{d}{\overline{w}_t}}{\norm{d}_2} \right )^B \right ).  
\end{equation*}
Dividing the actual distance of fire spread in direction $d$ over the first 24 hours after ignition by this multiplier implies a value for the coefficient $V$; we take the maximum value over all directions.  Finally, we rescale this value by the inverse of the average region multiplier $\mu_r$, weighted by area (see Section~\ref{subsec:regions}).  This gives a calibrated base spread rate of $V = 0.05 \text{ mi/hr}$.
The small magnitude of this value suggests that most of the fire spread is driven by wind; in fact, at a wind speed of $20 \text{ mi/hr}$, the maximum fire spread rate parallel to the wind direction is approximately $2 \text{ mi/hr}$ under these parameters.  As the exponent $B$ is close to $1$ and $C\overline{W}^B \gg 1$, this parameterization makes the ball spread sets tighter than the inner product spread sets.  For the ball spread sets, we use $B=1$, but maintain the calibrated values of $C$ and $V$.  Experiments across all three fires use the same calibrated value of $V$; differences in the rate of spread between the fires are modeled by rate of spread regions, which are computed as described in Section~\ref{subsec:regions}.

\subsection{Data-Driven Rate of Spread Regions}
\label{subsec:regions}

To adjust the base rate of spread $V$ by location, we scale proportionally by the \emph{wildland fire potential index} (WFPI; \citealt{burgan1998fuel,usgs2025wfpi}).  The WFPI is a raster dataset, composed of a grid with a value assigned to each cell, published daily at a spatial resolution of 1 km.  The WFPI value is a measure of vegetation flammability, computed from inputs that summarize the moisture, temperature, and fuel mixture at a location.  The WFPI correlates with the occurrence of large fires and with the proportion of small fires that become large \citep{preisler2009forecasting}, and has been used as a metric of fire risk in power grid operations \citep{taylor2022framework, piansky2024long, greenough2025wildfire}.  We use the WFPI as a proxy for the rate of fire spread, relative to the calibrated baseline value $V$.

We take the WFPI dataset for the ignition date of each fire and rescale the values to remove the impact of wind, using the function $\nu \rightarrow (1 + 0.6s/35)^{-1} \nu$,
where $s$ is the wind speed in knots \citep{usgs2025wfpi}.  This rescaling prevents the impact of wind from being counted twice, as the wind velocity is also incorporated into the spread sets $\mathcal{S}(w,x)$.  For rescaling, we use the 24-hour average wind speed $\norm{\overline{w}_t}_2$ from the low flexibility experimental setting.  The rescaled WFPI values fall on the interval $[0,100]$, where higher values indicate higher fire potential.

We aggregate the WFPI raster grid into the polyhedral regions $\mathcal{P}_r$ by training an optimal classification tree to predict the WFPI value associated with a coordinate pair.  Classification trees categorize data by a series of affine rules, then assign a label to the data according to the assigned cluster, in this case, the average of the values in the cluster.  After training, the clusters summarize the data, so that the error between the true value at each point and the label assigned by the clusters is minimized.  Each cluster, indexed by $r \in \mathcal{R}$, has an associated set of affine inequalities, which determine membership, and a scalar label.  This construction aligns with our desired properties for the rate of spread regions: the inequalities define a polyhedral region $\mathcal{P}_r$ and have an associated value, denoted $f_r$.  We define the spread rate multipliers as the normalized label: $\mu_r = f_r / 100$.  Due to the structure of the classification tree, the regions $\{\mathcal{P}_r\}_{r \in \mathcal{R}}$ partition $\mathbb{R}^2$.  

We train trees by the tree alternating optimization (TAO) algorithm for bivariate classification trees \citep{kairgeldin2024bivariate}.  Bivariate classification trees consider no more than two features in each affine rule; however, as the WFPI coordinates are two-dimensional, this property is not restrictive.  The TAO algorithm sequentially optimizes subproblems at each node of the tree to minimize the regression error of the labels assigned to the clustered data.  The subproblems to assign labels compute averages over values in a cluster, and the subproblems to compute affine rules enumerate orientations (\ie, angles) of the rule in two dimensions and select the best rule by the downstream error of the implied classification.  We train trees to minimize the regularized mean squared regression error, where the regularization penalizes nodes with a nontrivial affine rule (\ie, where the data is actually split).  The training uses a regularization parameter of $25$ and a tree depth of $3$, yielding a maximum of $8$ regions ($|\mathcal{R}| \leq 8$).  We train 25 randomly initialized trees and keep the single tree with the smallest error.

Figure~\ref{fig:WFPIRegions} compares the true WFPI data against the regions $\mathcal{P}_r$ and associated labels $f_r$ extracted from the trained classification trees, in the area around each of the Eaton, Palisades, and Park fires.  The regions effectively summarize high-level geographic trends in the WPFI dataset.

%%%%%%%%%%%%%%%%%% Figure %%%%%%%%%%%%%%%%%%
\begin{figure}[tp]
    \centering
    \includegraphics{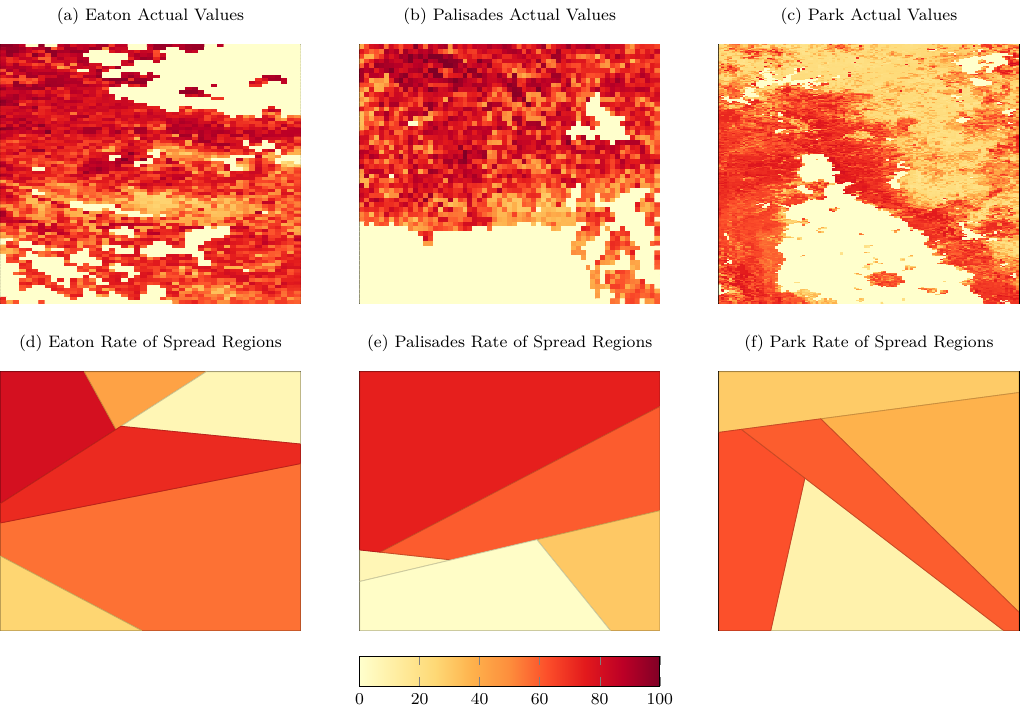}
    \caption{WFPI data (top) and rate of spread regions (bottom) extracted from optimal classification trees for each fire.}
    \label{fig:WFPIRegions}
\end{figure}
%%%%%%%%%%%%%%%%%% Figure %%%%%%%%%%%%%%%%%%

%%Computational Experiments and Results
\section{Computational Experiments and Results}
\label{sec:experiments}

This section details the results of our computational experiments.  In Section~\ref{subsec:mintimeresults}, we solve minimum time-to-outage fire spread models to screen the contingencies that are added to the optimal power flow model.  In Section~\ref{subsec:sequentialresults}, we construct a contingency sequence that incurs the largest amount of load shed.

For both experiments, we use the inputs described in Section~\ref{sec:data}.  Optimization models are solved by Gurobi~12.0.3, with feasibility and optimality tolerances of $10^{-5}$.   Tolerances are loosened from the default due to the complexity of the nonlinearities in the MICP~\eqref{abstractOpt} and the size of the DC-OPF~\eqref{DCOPF}.  The mixed-integer conic models~\eqref{abstractOpt} are solved with outer approximation of the conic constraints (\verb|MIQCPMethod| = 1), constraint aggregation disabled, and emphasis on improving the lower bound (\verb|MIPFocus| = 2).  Additionally, branching on the outage indicator variables $o_{et}$ is prioritized.  The DC-OPF~\eqref{DCOPF} is solved with the barrier method and with rescaling disabled (\verb|ScaleFlag| = 0).  Experiments are conducted in Python~3.10 on dual 48-core AMD EPYC 9474F@3.60 GHz processors with 386~GB of RAM.

\subsection{Contingency Screening for Optimal Power Flow}
\label{subsec:mintimeresults}

To screen contingencies, we solve the adversarial wildfire model \eqref{abstractOpt} for every subset of contingency elements $\overline{\mathcal{E}} \subseteq \mathcal{E}$ with the minimum time-to-outage objective~\eqref{eq:minTimeObj}.  For each subset, we solve a model with the ball spread sets $\mathcal{S}^{\circ}$ and with the inner product spread sets $\mathcal{S}^{\langle \rangle}$, and in the low flexibility and high flexibility settings.  Table~\ref{tbl:minTimeSummary} reports the number of variables and constraints by the number of contingency elements, and gives summary statistics for the model solve time.  These statistics summarize performance across the three fires and all contingency subsets of a fixed size.  The inner product spread sets require more variables and constraints than the ball sets, due to the reformulation of the power cone with second-order cones.  All models solve in under 7~minutes, with average solve times under 2~minutes.  There is no consistent trend in solution time across the spread sets and the flexibility setting, but solve times generally increase as the number of contingency elements increases.

%%%%%%%%%%%%%%%%%% Figure %%%%%%%%%%%%%%%%%%
\begin{table}[tp]
    \centering
    \begin{tabular}{@{}ccrrrrrrrrrr@{}}
    \toprule
     &  & \multicolumn{4}{c}{} & \multicolumn{6}{c}{Solve Time (s)} \\
    \cmidrule(lr){7-12}	
     &  & \multicolumn{2}{c}{\# Variables} & \multicolumn{2}{c}{\# Constraints} & \multicolumn{3}{c}{LF} & \multicolumn{3}{c}{HF} \\
    \cmidrule(lr){3-4}	\cmidrule(lr){5-6}	\cmidrule(lr){7-9}	\cmidrule(lr){10-12}	
     &  & Binary & All & Conic & All & Min & Mean & Max & Min & Mean & Max \\
    Spread Set & $|\overline{\mathcal{E}}|$ &  &  &  &  &  &  &  &  &  &  \\
    \midrule
    \multirow[c]{5}{*}{$\mathcal{S}^{\circ}$} & 1 & 168 & 1,007 & 47 & 2,627 & 0.1 & 0.2 & 0.5 & 0.0 & 0.1 & 0.1 \\
     & 2 & 312 & 1,894 & 70 & 5,159 & 0.1 & 0.6 & 2.1 & 0.1 & 4.7 & 59.5 \\
     & 3 & 456 & 2,781 & 93 & 7,691 & 0.3 & 1.2 & 3.6 & 0.2 & 25.4 & 180.2 \\
     & 4 & 600 & 3,668 & 116 & 10,223 & 0.4 & 2.0 & 6.8 & 3.6 & 72.9 & 399.2 \\
     & 5 & 744 & 4,555 & 139 & 12,755 & 0.7 & 2.6 & 5.0 & 4.5 & 79.8 & 229.1 \\
    \cline{1-12}
    \multirow[c]{5}{*}{$\mathcal{S}^{\langle \rangle}$} & 1 & 168 & 2,896 & 611 & 3,557 & 0.1 & 1.6 & 6.1 & 0.0 & 0.1 & 0.2 \\
     & 2 & 312 & 5,624 & 1,150 & 6,947 & 0.3 & 7.6 & 39.4 & 0.1 & 1.4 & 7.6 \\
     & 3 & 456 & 8,352 & 1,689 & 10,337 & 0.5 & 31.9 & 284.8 & 0.6 & 4.6 & 27.7 \\
     & 4 & 600 & 11,080 & 2,228 & 13,727 & 1.4 & 56.7 & 211.3 & 1.0 & 8.1 & 32.9 \\
     & 5 & 744 & 13,808 & 2,767 & 17,117 & 2.5 & 70.5 & 127.5 & 3.2 & 16.0 & 39.9 \\
    \bottomrule
    \end{tabular}
    \caption{Summary statistics for minimum time-to-outage outage models.  Number of variables and constraints are reported for the Eaton fire with low flexibility.}
    \label{tbl:minTimeSummary}
\end{table}
%%%%%%%%%%%%%%%%%% Figure %%%%%%%%%%%%%%%%%%

Figure~\ref{fig:optimalOutageTime} shows the distribution of minimum time-to-outage values over subsets $\overline{\mathcal{E}} \subseteq \mathcal{E}$ of the contingency elements, stratified by the cardinality $k$ of the subset, high and low flexibility setting, and choice of spread set.  As is expected, minimum outage times are larger in the low flexibility (LF) setting than the high flexibility (HF) setting, meaning the fire takes longer to reach all contingency elements when the ignition point is fixed and wind uncertainty is small.  The outage times generally increase with the cardinality $k$ of the contingency, as there are more elements to be reached by fire.  Outage times are also larger for the ball spread set than the inner product spread set, as the ball spread set is tighter than the inner product spread set under our model calibration, as discussed in Sections~\ref{subsec:chooserelax}~and~\ref{subsec:calibration}.  In some settings, the largest outage time ($k = 5$ and $\overline{\mathcal{E}} = \mathcal{E}$) exceeds the outage times for all single-element contingencies ($|\overline{\mathcal{E}}| = 1$).  This is due to the interaction of the fire spread with the wind vector and ignition point: these variables impact the spread towards all contingency elements, increasing the minimum outage time when multiple elements are modeled simultaneously.  This result demonstrates the significance of modeling the correlative impact of the wind and ignition site in multi-element contingencies.

%%%%%%%%%%%%%%%%%% Figure %%%%%%%%%%%%%%%%%%
\begin{figure}[tp]
    \centering
    \includegraphics{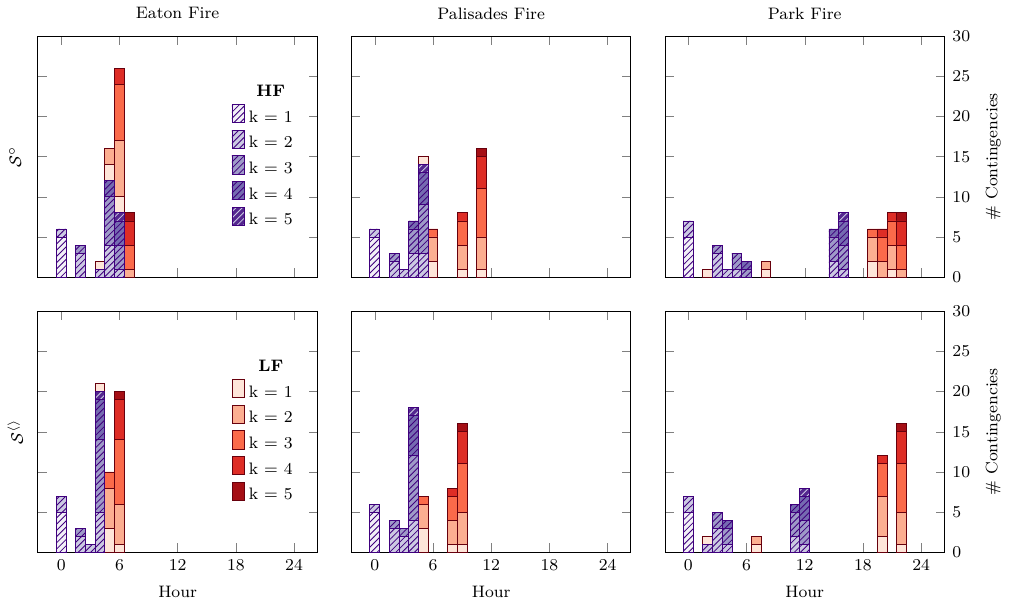}
    \caption{Distribution of minimum time-to-outage values by fire, spread set, flexibility, and number of contingency elements.}
    \label{fig:optimalOutageTime}
\end{figure}
%%%%%%%%%%%%%%%%%% Figure %%%%%%%%%%%%%%%%%%

Figure~\ref{fig:exampleSol} shows the optimal wind vectors and fire point locations for a minimum time-to-outage solution of the Eaton fire with high flexibility and the ball spread set, for a contingency with all five elements in $\mathcal{E}$.  Figure~\hyperref[fig:exampleSol]{\ref*{fig:exampleSol}a} shows the optimal wind vectors by period, which blow from the north in the first four hours, then from the west in the fifth hour, with speeds near the upper bound $\varepsilon$.  Figure~\hyperref[fig:exampleSol]{\ref*{fig:exampleSol}b} plots the optimal fire locations $x_{et}$, which track fire spread over time.  The ignition point is north of most contingency elements, and the fire spreads to the southeast and southwest from the ignition point, assisted by the north winds.  In the final period, the fire has nearly reached the elements to the west, so the west wind encourages greater fire spread to the east, allowing the fire to reach the remaining elements.

%%%%%%%%%%%%%%%%%% Figure %%%%%%%%%%%%%%%%%%
\begin{figure}[tp]
    \centering
    \begin{subfigure}[t]{0.4\textwidth}
        \centering
        \includegraphics[valign=t]{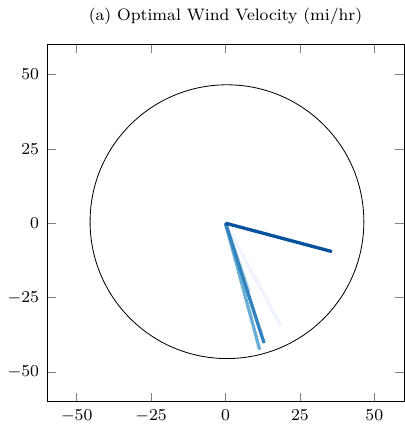}
    \end{subfigure}
    \hspace{1.5em}
    \begin{subfigure}[t]{0.5\textwidth}
        \centering
        \includegraphics[valign=t]{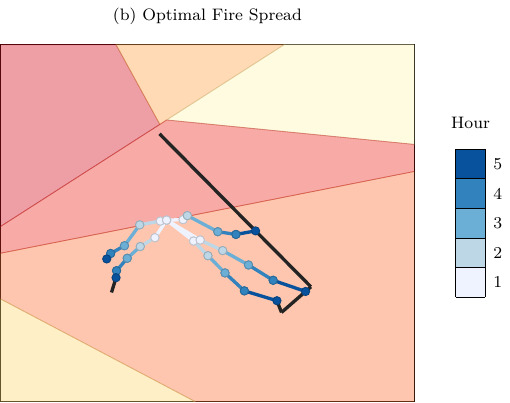}
    \end{subfigure}
    \caption{An optimal solution for the 5-element contingency of the Eaton fire with high flexibility and the ball spread set.  (a)~Hourly optimal wind velocity. (b)~Hourly pointwise fire spread (blue) reaches the power lines (black) in the contingency; rate of spread regions from Figure~\ref{fig:WFPIRegions} are shown.}
    \label{fig:exampleSol}
\end{figure}
%%%%%%%%%%%%%%%%%% Figure %%%%%%%%%%%%%%%%%%

After evaluating the minimum time-to-outage, we select contingencies to include in the sets $\mathcal{K}$ and solve the DC-OPF problem \eqref{DCOPF}.  We compare three regimes for contingency selection: including no contingencies, including every contingency at every time, and thresholding, where contingencies are included after their minimum time-to-outage.  When including no contingencies, $\mathcal{K} = \emptyset$, and when including all contingencies, $\mathcal{K} = \mathcal{T} \times 2^{\mathcal{E}}$.  To threshold contingencies, let $t^*(\overline{\mathcal{E}})$ be the minimum time-to-outage of contingency $\overline{\mathcal{E}}$, computed by solving \eqref{abstractOpt}.  Then, $\mathcal{K} = \{(t,\overline{\mathcal{E}}) \in \mathcal{T} \times 2^{\mathcal{E}} \,:\, t \geq t^*(\overline{\mathcal{E}}) \}.$  In effect, thresholding only adds contingencies in periods where they feasibly can occur by our fire spread model.  Table~\ref{tbl:opfScreening} shows the results of these experiments for the ball spread set with low flexibility.  We first observe that thresholding contingencies results in many fewer contingencies included, relative to including all contingencies.  Due to the reduction in model size, solve times are significantly reduced, by over 3x for the Park fire.  The contingency selection regimes do not drastically change the base operational cost $O^{\mathrm{n}}(p)$, although the costs are slightly lower for the cases with thresholded or no contingencies.  As the wildfire impacts are localized but the DC-OPF models the entirety of California, any changes in base cost are likely dominated by the scale of the cost across the entire system.  The base objective in the Park fire is relatively large due to load shedding in the base operational decision, which does not occur in the Eaton and Palisades models.  

The contingency screening strategies differ in the amount of post-contingency load shed (\ie, contingency costs).  We report the load shed in all contingencies (\ie, $\sum_{\overline{\mathcal{E}} \subseteq \mathcal{E},\ t \in \mathcal{T}} O^{\mathrm{ctg}}_{\overline{\mathcal{E}}t}(p)$) and only from contingencies in the thresholded set (\ie, $\sum_{\overline{\mathcal{E}} \subseteq \mathcal{E},\ t \geq t^*(\overline{\mathcal{E}})} O^{\mathrm{ctg}}_{\overline{\mathcal{E}}t}(p)$).  Note that these metrics give the total load shed over multiple possible contingencies, which may have overlapping outages.  Relative to the no contingency baseline, adding thresholded contingencies and all contingencies yields essentially the same reduction in load shedding on the contingencies in the thresholded set, reducing load shed by over 97\% for the Park fire and 70\% for the Eaton fire.  Post-contingency load shed is never eliminated, as this may require costly dispatch or be infeasible due to congestion and islanding.  Including all contingencies differs from the thresholding approach in the scale of load shed from all contingencies, indicating that load shedding for the additional contingencies can be alleviated by slightly increasing the base cost and exerting additional computational effort.  However, as the fire spread models suggest that these additional contingencies cannot reasonably be caused by wildfire, the additional cost and effort is unnecessary to mitigate wildfire impacts on the power grid.

%%%%%%%%%%%%%%%%%% Figure %%%%%%%%%%%%%%%%%%
\begin{table}[tp]
    \centering
    \begin{tabular}{@{}ccrrrrr@{}}
    \toprule
     &  & \multicolumn{3}{c}{} & \multicolumn{2}{c}{Load Shed (MWh)} \\
    \cmidrule(lr){6-7}	
     &  & $|\mathcal{K}|$ & Solve Time (s) & Base Cost (\$/hr) & Threshold & All \\
    Fire & Contingencies &  &  &  &  &  \\
    \midrule
    \multirow[c]{3}{*}{Eaton} & All & 744 & 276.4 & 644,355 & 11,902 & 17,683 \\
     & Threshold & 556 & 230.3 & 636,464 & 11,902 & 24,685 \\
     & None & 0 & 1.8 & 628,926 & 39,593 & 54,188 \\
    \cline{1-7}
    \multirow[c]{3}{*}{Palisades} & All & 744 & 715.1 & 627,704 & 38,868 & 70,256 \\
     & Threshold & 455 & 391.3 & 627,102 & 38,868 & 70,612 \\
     & None & 0 & 1.8 & 626,201 & 39,565 & 71,540 \\
    \cline{1-7}
    \multirow[c]{3}{*}{Park} & All & 744 & 274.0 & 10,787,083 & 28 & 309 \\
     & Threshold & 148 & 71.3 & 10,784,391 & 65 & 9,316 \\
     & None & 0 & 1.6 & 10,781,720 & 2,719 & 20,359 \\
    \bottomrule
    \end{tabular}
    \caption{Solution of DC-OPF by contingency screening strategy.  Solution time, base operational costs, and post-contingency load shed are compared.  Experiments are shown for the ball spread set with low flexibility.}
    \label{tbl:opfScreening}
\end{table}
%%%%%%%%%%%%%%%%%% Figure %%%%%%%%%%%%%%%%%%

\subsection{Sequential Contingency Design}
\label{subsec:sequentialresults}

We also run experiments to generate a sequence of element outages that maximize the amount of load shed over the horizon $T$.  To compute coefficients, we solve the operational DC-OPF \eqref{DCOPF} with no contingencies ($\mathcal{K} = \emptyset$) to optimal solution $p^*$.  We then assign objective weights to every subset of contingency elements in every period according to the amount of load shed incurred by that outage: $c_t(\overline{\mathcal{E}}) = O^{\mathrm{ctg}}_{\overline{\mathcal{E}}t}(p^*).$
Under these weights, the optimal objective value of \eqref{abstractOpt} gives the largest amount of load shed that can be caused by a wildfire that sequentially outages power lines.

Table~\ref{tbl:sequentialSummary} gives the number of variables and constraints for these models.  Compared to the minimum time-to-outage models (Table~\ref{tbl:minTimeSummary}), these models contain additional binary variables to generate indicators for each subset of elements, as in the reformulation~\eqref{eq:constrOutageComb}.  Due to differences in the number of rate of spread regions $|\mathcal{R}|$, the model size differs by fire.

%%%%%%%%%%%%%%%%%% Figure %%%%%%%%%%%%%%%%%%
\begin{table}[tp]
    \centering
    \begin{tabular}{@{}ccrrrr@{}}
    \toprule
     &  & \multicolumn{2}{c}{\# Variables} & \multicolumn{2}{c}{\# Constraints} \\
    \cmidrule(lr){3-4}	\cmidrule(lr){5-6}	
     &  & Binary & All & Conic & All \\
    Fire & Spread Set &  &  &  &  \\
    \midrule
    \multirow[c]{2}{*}{Eaton} & $\mathcal{S}^{\circ}$ & 1,608 & 5,419 & 139 & 16,711 \\
     & $\mathcal{S}^{\langle \rangle}$ & 1,608 & 14,672 & 2,767 & 21,073 \\
    \cline{1-6}
    \multirow[c]{2}{*}{Palisades} & $\mathcal{S}^{\circ}$ & 1,488 & 4,819 & 139 & 14,761 \\
     & $\mathcal{S}^{\langle \rangle}$ & 1,488 & 13,727 & 2,652 & 19,353 \\
    \cline{1-6}
    \multirow[c]{2}{*}{Park} & $\mathcal{S}^{\circ}$ & 1,488 & 4,819 & 139 & 14,761 \\
     & $\mathcal{S}^{\langle \rangle}$ & 1,488 & 13,727 & 2,652 & 19,353 \\
    \bottomrule
    \end{tabular}
    \caption{Number of variables and constraints in the sequential contingency models by spread set and fire, under the low flexibility setting.}
    \label{tbl:sequentialSummary}
\end{table}
%%%%%%%%%%%%%%%%%% Figure %%%%%%%%%%%%%%%%%%

Table~\ref{tbl:sequentialSolution} reports the solve time and optimal objective of these models.  In this table, the optimal load shed is the optimal objective value, whereas the base load shed gives the amount of load shed if all contingency elements $\mathcal{E}$ are outaged at the minimum outage time $t^*(\mathcal{E})$; that is, 
$\sum_{t \geq t^*(\mathcal{E})} O^{\mathrm{ctg}}_{\mathcal{E}t}(p^*)$.
In these experiments, the inner product spread sets require more computation time: the inner product models solve within 25~minutes, whereas the ball models solve within 2~minutes.  In the low flexibility setting, the ball spread sets give smaller objective values than the inner product sets; however, this trend reverses in the high flexibility setting.  Although the ball spread sets are generally tighter than the inner product sets, due to the rounding of exponent $B$ to $1$, the maximum spread rate is larger for the ball spread sets than the inner product sets.  This allows for the ball spread sets to produce larger optimal objective values in some cases.  Considering a sequence of outages (the optimal load shed) results in larger load shed than considering simultaneous outages (the base load shed).  For the Park fire, the load shed is up to 21x larger, and we find increases of up to 19\% for the other fires.  These results emphasize the importance of considering outage sequences when analyzing the impact of wildfire on the power grid.

%%%%%%%%%%%%%%%%%% Figure %%%%%%%%%%%%%%%%%%
\begin{table}[tp]
    \centering
    \begin{tabular}{@{}cccrrr@{}}
    \toprule
     &  &  &  & \multicolumn{2}{c}{Load Shed (MWh)} \\
    \cmidrule(lr){5-6}	
     &  &  & Solve Time (s) & Base & Optimal \\
    Fire & Spread Set & Flexibility &  &  &  \\
    \midrule
    \multirow[c]{4}{*}{Eaton} & \multirow[c]{2}{*}{$\mathcal{S}^{\circ}$} & HF & 13.7 & 2,173 & 2,591 \\
     &  & LF & 4.1 & 2,031 & 2,236 \\
    \cline{2-6}
     & \multirow[c]{2}{*}{$\mathcal{S}^{\langle \rangle}$} & HF & 72.1 & 2,406 & 2,474 \\
     &  & LF & 25.0 & 2,173 & 2,295 \\
    \cline{1-6}
    \multirow[c]{4}{*}{Palisades} & \multirow[c]{2}{*}{$\mathcal{S}^{\circ}$} & HF & 15.3 & 3,428 & 3,961 \\
     &  & LF & 106.7 & 2,239 & 2,617 \\
    \cline{2-6}
     & \multirow[c]{2}{*}{$\mathcal{S}^{\langle \rangle}$} & HF & 20.6 & 3,602 & 3,960 \\
     &  & LF & 1,357.7 & 2,655 & 2,951 \\
    \cline{1-6}
    \multirow[c]{4}{*}{Park} & \multirow[c]{2}{*}{$\mathcal{S}^{\circ}$} & HF & 14.0 & 249 & 935 \\
     &  & LF & 7.2 & 38 & 797 \\
    \cline{2-6}
     & \multirow[c]{2}{*}{$\mathcal{S}^{\langle \rangle}$} & HF & 40.1 & 465 & 923 \\
     &  & LF & 1,027.1 & 38 & 797 \\
    \bottomrule
    \end{tabular}
    \caption{Sequential contingency model solve time and optimal load shed over 24 hours, compared against base load shed from outaging all five elements simultaneously at the minimum outage time.}
    \label{tbl:sequentialSolution}
\end{table}
%%%%%%%%%%%%%%%%%% Figure %%%%%%%%%%%%%%%%%%

Figure~\ref{fig:exampleSequential} depicts the number of outaged elements by period in the optimal outage sequence.  Results are shown for the high flexibility setting and the ball spread set.  The sequence prioritizes outaging some elements in early periods, while other elements are outaged closer to the end of the horizon.  In all cases, the fire ignition point is on the geometry of some power line, outaging one element in the first period.  For the Eaton fire, the next three elements are outaged in the sixth period, and the final element in period 14.  In these solutions, elements with a large impact on the grid are outaged early in the sequence, while elements with lesser impact have lower priority and are reached by the fire later.  For all fires, the outage time of the fifth element is always later than the minimum time-to-outage of the contingency (see Figure~\ref{fig:optimalOutageTime}), as the optimal sequential solution selects wind vectors and an ignition point which outage elements that cause the most load shed early, at the cost of outaging less impactful elements after the minimum outage time.

%%%%%%%%%%%%%%%%%% Figure %%%%%%%%%%%%%%%%%%
\begin{figure}[tp]
    \centering
    \includegraphics{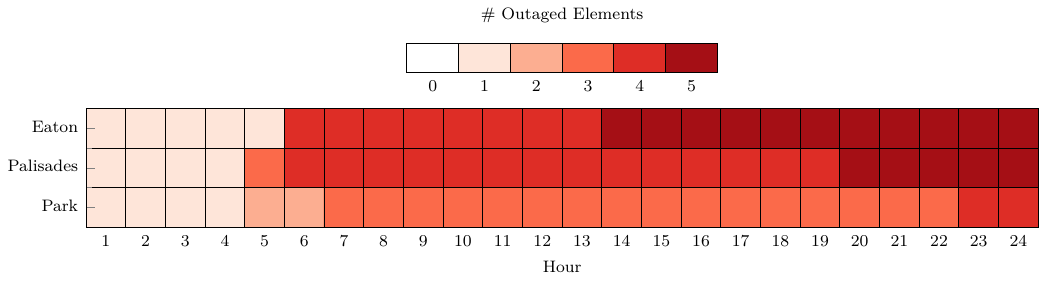}
    \caption{Number of outaged elements by hour in optimal sequential contingency solutions, in the high flexibility setting with the ball spread set.}
    \label{fig:exampleSequential}
\end{figure}
%%%%%%%%%%%%%%%%%% Figure %%%%%%%%%%%%%%%%%%

%%Conclusion
\section{Conclusion}
\label{sec:conclusion}

In this work, we propose an optimization model to describe adversarial wildfires that impart the worst-case impact on power grid infrastructure.  Inspired by the Rothermel model, we propose fire spread sets and convex conic relaxations that are compatible with mixed-integer conic programs.  These relaxations include a novel second-order conic constraint that relaxes the inner product over Euclidean balls.  We design three experiments that align with California wildfires and evaluate the minimum time-to-outage of multi-element contingencies and the maximum load shed of a sequence of element outages.  We apply the minimum time-to-outage objectives to screen contingencies for a security-constrained optimal power flow model, eliminating unrealistic contingencies while encouraging operational resilience of the power grid against wildfire.

This perspective on modeling adversarial wildfires represents a unique methodology for characterizing the worst-case fire behavior.  Although we apply this approach specifically to power grid analysis and operations, it is generalizable to many areas of research which consider the impact of wildfire, including wildfire prevention, suppression, and threat evaluation for population centers or other infrastructure.  Tight integration of our adversarial model into primary robust optimization models in a variety of settings will help promote decisions with greater resilience against worst-case wildfires.

\section*{Acknowledgement}
This work was performed under the auspices of the U.S. Department of Energy by Lawrence Livermore National Laboratory under Contract DE-AC52-07NA27344 and was supported in part by the LLNL LDRD Program under Project 25-SI-007.  M. Brun and X.A. Sun were supported in part by the Sloan Foundation.  The authors thank Nai Chiang, Juliette Franzman, Elizabeth Glista, Minda Monteagudo, Matthew Signorotti, and Tomas Valencia Zuluaga at Lawrence Livermore National Laboratory for their help with preparation of data and code.

\bibliographystyle{informs2014}
\bibliography{AdversarialWildfire.bib}

\newpage
%%Appendix
\appendix
\section{Appendix: Complete Formulation of DC-OPF}
\label{app:explictOPF}
In this appendix, we provide a mathematical description of the feasible sets and objective functions that compose the security-constrained DC-OPF~\eqref{DCOPF}.  This model is the operational portion of the capacity expansion model of \citeappendix{musselman2025climateappendix}.  Let $\mathcal{I}$ be the set of buses in the power grid and $\mathcal{J}$ be the set of lines, where $(i,j) \in \mathcal{J}$ represents a line connecting bus $i$ to bus $j$.  The set $\mathcal{G}$ gives the types of generation considered (\eg, natural gas, wind, hydroelectric).  We consider a multi-period model over periods $\mathcal{T}$, the same horizon as considered in the adversarial wildfire model~\eqref{abstractOpt}.

In the base operational model, let the variable $p^{\mathrm{g}}_{igt}$ denote the quantity of generation of type $g$ at bus $i$ in period $t$, the variable $p^{\mathrm{b}}_{it}$ denote the amount of power in storage, and the variable $p^{\mathrm{s}}_{it}$ denote the quantity of load shed.  We also introduce variables $p^{\Delta}_{it+}$ and $p^{\Delta}_{it-}$ for the amount of power charged and discharged, respectively, by storage.  These variables are constrained as follows:
\begin{subequations}
    \label{eq:opfconstrOperation}
    \begin{align}
        & p^{\mathrm{g}}_{igt} \leq \overline{P}^{\mathrm{g}}_{igt} \quad && \forall i \in \mathcal{I},\ g \in \mathcal{G},\ t \in \mathcal{T}, \label{eq:opfconstrPmax}\\
        & \sum_{t \in \mathcal{T}} p^{\mathrm{g}}_{igt} \leq \overline{H}^{\mathrm{g}}_{ig} \quad && \forall i \in \mathcal{I},\ g \in \mathcal{G},\label{eq:opfconstrPTotal}\\
        & p^{\mathrm{b}}_{it} \leq \overline{P}^{\mathrm{b}}_{it} \quad && \forall i \in \mathcal{I},\ t \in \mathcal{T},\label{eq:opfconstrBattmax}\\
        & p^{\Delta}_{it+} \leq \overline{P}^{\Delta}_{it},\ p^{\Delta}_{it-} \leq \overline{P}^{\Delta}_{it} \quad && \forall i \in \mathcal{I},\ t \in \mathcal{T},\label{eq:opfconstrBattChargeMax}\\ 
        & p^{\mathrm{b}}_{it} = p^{\mathrm{b}}_{i,t-1} + p^{\Delta}_{it+} - p^{\Delta}_{it-} \quad && \forall i \in \mathcal{I},\ t \in \mathcal{T},\label{eq:opfconstrBattSOC}\\
        & p^{\mathrm{b}}_{i0} = p^{\mathrm{b}}_{iT} \quad && \forall i \in \mathcal{I},\label{eq:opfconstrBattBoundary}\\
        & p^{\mathrm{s}}_{it} \leq P^{\mathrm{d}}_{it} \quad && \forall i \in \mathcal{I},\ t \in \mathcal{T},\label{eq:opfconstrLSMax}\\
        & p \geq 0.
    \end{align}
\end{subequations}
Constraint~\eqref{eq:opfconstrPmax} enforces an upper bound $\overline{P}^{\mathrm{g}}_{igt}$ on the power generation of type $g$, and constraint~\eqref{eq:opfconstrPTotal} bounds the total energy produced over the model horizon by parameter $\overline{H}^{\mathrm{g}}_{ig}$, which may differ from the per-period generation limit due to resource constraints, especially in the case of hydroelectric power.  Constraints~\eqref{eq:opfconstrBattmax}~and~\eqref{eq:opfconstrBattChargeMax} limit the amount of stored energy by parameter $\overline{P}^{\mathrm{b}}_{it}$ and the rate of charge and discharge by parameter $\overline{P}^{\Delta}_{it}$, respectively.  Constraint~\eqref{eq:opfconstrBattSOC} models the change in stored energy across periods, and constraint~\eqref{eq:opfconstrBattBoundary} requires that the end-of-horizon stored power is the same as in the first period, preventing initial storage from being treated as free energy.  Finally, constraint~\eqref{eq:opfconstrLSMax} requires that the load shed does not exceed the load $P^{\mathrm{d}}_{it}$.

To model transmission, we introduce variables $f_{ijt}$ for the power flow on the line from bus $i$ to bus $j$ and control variables $\theta_{it}$ for the voltage angle at bus $i$, subject to the following constraints:
\begin{subequations}
    \label{eq:opfconstrPowerFlow}
    \begin{align}
        & f_{ijt} = \frac{1}{\beta_{ij}} (\theta_{jt} - \theta_{it}) \quad && \forall (i,j) \in \mathcal{J},\ t \in \mathcal{T},\label{eq:opfconstrFlowDef}\\
        & |f_{ijt}| \leq \overline{f}_{ij} \quad && \forall (i,j) \in \mathcal{J},\ t \in \mathcal{T}\label{eq:opfconstrFlowMax}
    \end{align}
\end{subequations}
Constraint~\eqref{eq:opfconstrFlowDef} defines the flow on each line under the DC model, where $\beta_{ij}$ gives the reactance on the line from bus $i$ to bus $j$.  Constraint~\eqref{eq:opfconstrFlowMax} bounds the magnitude of the line power flows by a flow limit $\overline{f}_{ij}$.  Finally, we enforce nodal power balance:
\begin{equation}
    \label{eq:opfconstrBalance}
    \begin{aligned}
        & \mathrlap{\sum_{g \in \mathcal{G}} p^{\mathrm{g}}_{igt} + \eta^- p^{\Delta}_{it-} + p^{\mathrm{s}}_{it} + \sum_{j \in \mathcal{J}^+_i} f_{ijt}}\\
        & \quad = P^{\mathrm{d}}_{it} + \frac{1}{\eta^+} p^{\Delta}_{it+} + \sum_{j \in \mathcal{J}^-_i} f_{ijt} \quad && \forall i \in \mathcal{I},\ t \in \mathcal{T}. 
    \end{aligned}
\end{equation}
The set $\mathcal{J}^{+}_i := \{j \in \mathcal{I} \,:\, (j,i) \in \mathcal{J}\}$ gives the lines that end at bus $i$, and $\mathcal{J}^-_i$ is defined analogously for lines that start at bus $i$.  Parameters $(\eta^+,\eta^-) \subset (0,1]$ give the charging and discharging efficiency of storage, respectively.  The feasible set for base operations is $\mathcal{F}^{\mathrm{b}} := \{(p,f,\theta) \,:\, \eqref{eq:opfconstrOperation},\ \eqref{eq:opfconstrPowerFlow},\ \eqref{eq:opfconstrBalance}\}$.

The cost of base operations is due to the cost of generation, storage discharge, and load shed: 
\begin{equation}
    O^{\mathrm{b}}(p) := \sum_{t \in \mathcal{T}} \sum_{i \in \mathcal{I}} \left ( c^{\mathrm{b}} p^{\Delta}_{it-} + c^{\mathrm{s}} p^{\mathrm{s}}_{it} +  \sum_{g \in \mathcal{G}} c^{\mathrm{g}}_{g} p^{\mathrm{g}}_{igt} \right ),
\end{equation}
where $c^{\mathrm{g}}_{g}$ gives the per-unit cost of generation, $c^{\mathrm{b}}$ the per-unit cost of storage discharge, and $c^{\mathrm{s}}$ the per-unit cost of load shed.  Storage costs are due to operation and maintenance costs of pumped hydroelectric storage.

We consider a contingency where power lines $\mathcal{E} \subset \mathcal{J}$ are outaged in period $t \in \mathcal{T}$.  Although this model is limited to line outages, it can easily be generalized to generation, storage, or bus outages.  We introduce variables $p^{\mathrm{ctg, g}}_i$ and $p^{\mathrm{ctg,s}}_i$ for the post-contingency generation and load shed, respectively, and control variables $f^{\mathrm{ctg}}_{ij}$ and $\theta^{\mathrm{ctg}}_{i}$ for the post-contingency line flows and voltage angles.  In the contingency, the generation $p^{\mathrm{ctg,g}}_i$ includes the injections from generators and storage discharge, and load shed includes reductions in demand served and in the charging of storage, relative to the base solution.  The following constraints govern post-contingency operations:
\begin{subequations}
    \label{eq:ctgConstr}
    \begin{align}
        & p^{\mathrm{s}}_{it} \leq p^{\mathrm{ctg, s}}_i \leq P^{\mathrm{d}}_{it} + \frac{1}{\eta^+} p^{\Delta}_{it+} \quad && \forall i \in \mathcal{I},\label{eq:ctgconstrShed}\\
        & 0 \leq p^{\mathrm{ctg,g}}_i \leq \sum_{g \in \mathcal{G}} p^{\mathrm{g}}_{igt} + \eta^- p^{\Delta}_{it-} && \forall i \in \mathcal{I},\label{eq:ctgconstrGen}\\
        & f^{\mathrm{ctg}}_{ij} = \frac{1}{\beta_{ij}} (\theta^{\mathrm{ctg}}_{j} - \theta^{\mathrm{ctg}}_{i}) \quad && \forall (i,j) \in \mathcal{J} \setminus \mathcal{E},\label{eq:ctgconstrFlowDef}\\
        & f^{\mathrm{ctg}}_{ij} = 0 \quad && \forall (i,j) \in \mathcal{E},\label{eq:ctgconstrFlowOut}\\
        & |f^{\mathrm{ctg}}_{ij}| \leq \overline{f}_{ij} \quad && \forall (i,j) \in \mathcal{J},\label{eq:ctgconstrFlowMax}\\
        & \mathrlap{p^{\mathrm{ctg,s}}_{i} + p^{\mathrm{ctg,g}}_i + \sum_{j \in \mathcal{J}^+_i} f^{\mathrm{ctg}}_{ij}} \nonumber\\
        & \quad = P^{\mathrm{d}}_{it} + \frac{1}{\eta^+} p^{\Delta}_{it+} + \sum_{j \in \mathcal{J}^-_i} f^{\mathrm{ctg}}_{ij} \quad && \forall i \in \mathcal{I}.\label{eq:ctgconstrBalance}
    \end{align}
\end{subequations}
Constraint~\eqref{eq:ctgconstrShed} allows the post-contingency load shed to increase from the load shed in the base solution; it is bounded above by the total load and base charging of storage.  Constraint~\eqref{eq:ctgconstrGen} bounds the post-contingency injection due to generation and discharging of storage by the base injection from these devices.  The operational justification for these constraints is described in Section~\ref{sec:opf}.  Constraints \eqref{eq:ctgconstrFlowDef} and \eqref{eq:ctgconstrFlowOut} model post-contingency power flows on lines that are active and those that are outaged, respectively, and constraint~\eqref{eq:ctgconstrFlowMax} imposes the flow limit.  Constraint~\eqref{eq:ctgconstrBalance} enforces nodal power balance, where the total post-contingency injection due to generation and discharging of storage must equal the sum of load, charging of storage, and net outflow on lines, less any shedded load.  Shedding all load and setting generation to zero yields a feasible solution for the contingency, regardless of the base solution; therefore, this contingency model exhibits relatively complete recourse. We denote the post-contingency feasible set by $\mathcal{F}^{\mathrm{ctg}}_{\mathcal{E}t}(p) := \{(p^{\mathrm{ctg}},f^{\mathrm{ctg}},\theta^{\mathrm{ctg}}) \,:\, \eqref{eq:ctgConstr}\}$.  This set is written a function of the base operational decisions $p$ and parametrized by the outage set $\mathcal{E}$ and period $t$.  The constraints~\eqref{eq:ctgConstr} are linear in the post-contingency variables and in the base operational decision $p$.  The contingency objective function is due to the cost of shedding additional load, given by
\begin{equation}
    O^{\mathrm{s}}_{\mathcal{E}t}(p^{\mathrm{ctg}}) := \sum_{i \in \mathcal{I}} c^{\mathrm{s}} \left ( p^{\mathrm{ctg,s}}_i - p^{\mathrm{s}}_{it}\right ).
\end{equation}
For simplicity, we suppress the dependence of this function on the base operational decision $p$.  

Although this formulation of the DC-OPF model \eqref{DCOPF} permits simultaneous charging and discharging of storage, this behavior is limited in optimal solutions due to storage losses and the load shed cost associated with post-contingency reduction of charging.

\bibliographystyleappendix{informs2014}
\bibliographyappendix{AdversarialWildfire_appendix.bib}

\end{document}